\documentclass[11pt]{amsart}
\usepackage{latexsym}
\usepackage{aeguill}
\usepackage[latin1]{inputenc}
\usepackage[T1]{fontenc}
\usepackage{amsfonts}
\usepackage{amssymb}
\usepackage{xspace}
\usepackage{graphicx}
\usepackage{amsmath}
\usepackage{amsthm}
\usepackage{endnotes}
\usepackage{tikz}
\usepackage{pgf}
\usepackage{yfonts}

\usepackage[all]{xy}
\usepackage{amscd}
\usetikzlibrary{matrix}

\usepackage{comment}

\newcommand{\cal}{\mathcal}
\newcommand\<{\langle}
\renewcommand\>{\rangle}

\def\epsilon{\varepsilon}
\def\phi{\varphi}

\def\hat{\widehat}

\newcommand{\Out}{\mbox{Out}}

\def\strutdepth{\dp\strutbox}
\def \ss{\strut\vadjust{\kern-\strutdepth \sss}}
\def \sss{\vtop to \strutdepth{
\baselineskip\strutdepth\vss\llap{$\diamondsuit\;\;$}\null}}

\def\strutdepth{\dp\strutbox}
\def \sst{\strut\vadjust{\kern-\strutdepth \ssss}}
\def \ssss{\vtop to \strutdepth{
\baselineskip\strutdepth\vss\llap{$\spadesuit\;\;$}\null}}

\def\strutdepth{\dp\strutbox}
\def \ssh{\strut\vadjust{\kern-\strutdepth \sssh}}
\def \sssh{\vtop to \strutdepth{
\baselineskip\strutdepth\vss\llap{$\heartsuit\;\;$}\null}}

\def\qed{\hfill\rlap{$\sqcup$}$\sqcap$\par}
\def\bar{\overline}

\def\strutdepth{\dp\strutbox}
\def \ss{\strut\vadjust{\kern-\strutdepth \sss}}
\def \sss{\vtop to \strutdepth{
\baselineskip\strutdepth\vss\llap{$\diamondsuit\;\;$}\null}}

\def\strutdepth{\dp\strutbox}
\def \sst{\strut\vadjust{\kern-\strutdepth \ssss}}
\def \ssss{\vtop to \strutdepth{
\baselineskip\strutdepth\vss\llap{$\spadesuit\;\;$}\null}}

\def\qed{\hfill\rlap{$\sqcup$}$\sqcap$\par}

\newtheorem{thm}{Theorem}[section]
\newtheorem{cor}[thm]{Corollary}
\newtheorem{lem}[thm]{Lemma}
\newtheorem{prop}[thm]{Proposition}

\theoremstyle{definition}
\newtheorem{defn}[thm]{Definition}
\newtheorem{example}[thm]{Example}

\newtheorem{rem}[thm]{Remark}

\theoremstyle{remark}

\numberwithin{equation}{section}

\begin{document}

\author[K.~Ye]{Kaidi Ye}

\title[Quotient and Blow-up of automorphisms of Graphs of groups]
{Quotient and Blow-up of automorphisms of Graphs of groups}

\begin{abstract} 
In this paper we study the quotient and ``blow-up" of graph-of-groups $\cal{G}$ and of their automorphisms $H: \cal{G} \rightarrow \cal{G}$. We show that the existence of such a blow-up of any $\bar{H}: \bar{\cal{G}} \rightarrow \bar{\cal{G}}$, relative to a given family of ``local" graph-of-groups isomorphisms $H_{i}: \cal{G}_{i} \rightarrow \cal{G}_{i}$ depends crucially on the $H_{i}^{-1}$-conjugacy class of the correction term $\delta(E_i)$ for any edge $E_i$ of $\bar{\cal{G}}$, where $H$-conjugacy is a new but natural
concept introduced here.

As an application we obtain a criterion as to whether a partial Dehn twist can be blown up relative to local Dehn twists, to give an actual Dehn twist.
The results of this paper are also used crucially in the follow-up papers \cite{LY1, KY02, KY03}.
\end{abstract}


\maketitle

\section{Introduction}

${}^{}$
Graphs-of-groups and Bass-Serre theory have played a central role in geometric group theory ever since this field came into existence in the 1980's. As a prime example we would like to mention its prominent role in the understanding of automorphisms of a hyperbolic group $G$, see \cite{Levitt}, 
which is based on an essentially unique graph-of-groups decomposition of $G$, in case that $G$ is freely indecomposable.

If on the other hand the group $G$ is a free group $F_n$ of finite rank $n \geq 2$, then a special kind of graph-of-groups, called {\em very small}, plays an important role in the definition of the boundary of Culler-Vogtmann's Outer space $CV_n$, which is the analogue of Teichm\"uller space, for $\Out(F_n)$ in place of the mapping class group. The work presented here is mainly motivated by questions arising from this perspective, 
although we keep our set-up slightly more general.

\smallskip

Given a graph-of-groups $\cal{G}$ based on a finite connected graph $\Gamma=\Gamma(\cal{G})$, for any connected subgraph $\Gamma_0 \subset \Gamma$ we denote by $\cal{G}_{0}$ the restriction of $\cal{G}$ to $\Gamma_0$. There is a natural way to define a quotient graph-of-groups $\bar{\cal{G}}=\cal{G} \slash \cal{G}_{0}$ which is obtained by ``contracting" $\cal{G}_{0}$ into a single vertex $V_0$ with vertex group $G_{V_0} \cong \pi_{1}(\cal{G}_{0})$, thus giving rise to a canonical isomorphism 
$ \Theta: \pi_{1}(\cal{G}) \rightarrow \pi_{1}(\bar{\cal{G}})$.
By construction the quotient graph $\bar{\cal{G}}$ is {\it compatible} with the local graph-of-groups $\cal{G}_0$, in the sense that for any edge $E$ of $\bar{\cal{G}}$ with terminal vertex $V_0$ the canonical image of the edge group $G_{E}$ in the vertex group $G_{V_0} \cong \pi_{1}(\cal{G}_{0})$ is (up to conjugation) contained in one of the vertex groups of $\cal{G}_{0}$.

This quotient concept extends naturally to an isomorphism $H$ of $\cal{G}$ which acts as identity on the underlying graph 
$\Gamma$: We can 
construct a quotient graph-of-groups isomorphism $\bar{H}: \bar{\cal{G}} \rightarrow \bar{\cal{G}}$ which induces on the fundamental group an outer automorphism $\hat{\bar{H}}$ 
that is conjugate via $\Theta$ to the outer automorphism $\hat{H}$ induced by $H$,
as shown in the diagram below. 

\[
\begin{CD}
\pi_{1}(\cal{G}) @> \hat{H}>> \pi_{1}(\cal{G}) \\
@V\Theta VV @V \Theta VV \\
\pi_{1}(\bar{\cal{G}})@>>\hat{\bar{H}}> \pi_{1}(\bar{\cal{G}})
\end{CD}
\]
The restriction $H_{0}$ of $H$ to $\cal{G}_{0}$ is called the {\it local graph-of-groups isomorphism} at $V_0$, 
and again certain natural {\em compatibility conditions} between the pairs $(\bar H, \bar{\cal{G}})$ and $(H_0, \cal{G}_0)$ 
are 
satisfied, which are stated precisely in 
Definition \ref{defn-6.1} below.

Of course, both, 
the quotient graph-of-groups $\bar{\cal{G}}$ and the quotient isomorphism $\bar{H}$, 
are also well-defined modulo more than one 
pairwise 
disjoint connected sub-graph-of-groups $\cal{G}_{i}$ of $\cal{G}$.

\medskip

The main purpose of this paper is to 
study the converse of 
the above described quotient 
procedure, which 
we call 
the 
``blow-up" 
of a graph-of-groups isomorphism. We prove
(see Theorem \ref{blowup}):

\begin{thm}\label{intro1}
Let $\bar{H}: \bar{\cal{G}} \rightarrow \bar{\cal{G}}$ be a graph-of-groups isomorphism which acts as identity on the graph $\bar{\Gamma}$ underlying $\bar{\cal{G}}$. Assume that for 
some
vertices $V_i$ of $\bar{\Gamma}$ the group isomorphism $\bar{H}_{V_i}$ is induced by a 
local graph-of-groups 
automorphism 
$H_{i}: \cal{G}_{i} \rightarrow \cal{G}_{i}$ 
which also acts as identity on the underlying graph $\Gamma(\cal{G}_i)$.

Then one can blow up $(\bar{H}, \bar{\cal{G}})$ via the family of 
$(H_i, \cal G_i)$
to obtain a blow-up graph-of-groups 
isomorphism
$H: \cal G \to \cal G$, 
with induced outer automorphism $\hat H = \hat{\bar H}$, 
if and only if 
each 
$(H_i, \cal{G}_i)$ is compatible (in the sense of Definition \ref{defn-6.1}) with $(\bar{H}, \bar{\cal{G}})$.
\end{thm}

We are most interested 
in the special case where for any edge $E$ of $\bar{\cal G}$ the edge group $G_E$ 
is trivial. In this case 
the compatibility conditions from Definition \ref{defn-6.1} simplify to a property of the 
edge $E$ which
we call 
``locally zero''. Since this is a new concept, we will try to explain it here briefly:

Recall first that if $E$ terminates in the vertex $V_i$, then (as for any graph-of-groups isomorphism, see Definition \ref{graphofgroupsiso}) the {\em correction term} $\delta(E) \in G_{V_i}$ serves to make the edge and vertex isomorphisms $\bar H_{E}$ and $\bar H_{V}$ commute with the injective edge homomorphism $f_{E}: G_{E} \rightarrow G_{V_i}$.

Now, we say that 
$E$
is {\em locally zero} 
(see Definition \ref{locally-zero})
if the identification $G_{V_i} \cong \pi_1\cal G_i$ maps $\delta(E)$ to an element which is 
``$H_i^{-1}$-conjugate'' 
to an element that has $\cal{G}_{V_i}$-length equal to zero. If the local automorphism $H_i$ is equal to the identity map, then 
{\em $H_i^{-1}$-conjugation} 
will simply be the usual conjugation in $G_{V_i}$; in general though it is a more involved and quite delicate new notion, defined 
below in section~\ref{Hconj}.

\medskip 

In the last section of this paper we will apply Theorem \ref{intro1} to the case of Dehn twist automorphisms of a free group $F_{n}$.  
Classically, a {\it Dehn twist} $D=(\cal{G}, (z_{e})_{e \in E(\cal{G})})$ on a
graph-of-groups $\cal{G}$
is defined by a family of {\it twistors} $ (z_{e})_{e \in E(\cal{G})}$, where each $z_{e}$ is in the center of the edge group $G_{e}$ of $\cal{G}$. 
It turns out (see Proposition \ref{trivial edge groups}
and Remark \ref{example-Brian}) that for free groups an 
alternative, 
slightly more restrictive 
definition 
of a Dehn twist is given by graph-of-groups isomorphisms $H: \cal{G} \rightarrow \cal{G}$ where all edge groups of $\cal{G}$ are trivial and $H$ acts as identity on the underlying graph and on every vertex group of $\cal{G}$.

Inspired by this alternative definition, we define 
(see Definition~\ref{partial Dehn twist})
a {\it partial Dehn twist} 
$D: \cal{G} \rightarrow \cal{G}$ with $\pi_{1}(\cal{G}) \cong F_{n}$, {\em relative to some  family of vertices $V_{1}, \ldots, V_{m}$} of the underlying graph $\Gamma(\cal G)$, which differs from the above classical notion in that on these ``exceptional vertices'' $V_i$ the 
local automorphism $D_{V_i}: G_{V_i} \to G_{V_i}$ 
induced by $D$ may be non-trivial.

Of particular interest is the case where these non-trivial local automorphisms are all 
Dehn twist automorphisms 
themselves. This occurs naturally if one quotients a given Dehn twists modulo a family of pairwise disjoint subgraphs. The converse direction, however, is far less obvious, and the desired blow-up Dehn twist doesn't always exist. We prove here
(see Corollary~\ref{Dehn01}):

\begin{cor}
\label{intro2}
(1)
Let $\bar{D}: \bar{\cal{G}} \rightarrow \bar{\cal{G}}$ be a partial Dehn twist, and assume that for some family of vertices $V_i$ of $\bar{\cal{G}}$ the vertex group 
automorphisms 
$D_{V_i}$ are induced by 
Dehn twists 
$D_{V_i}: \cal{G}_{V_i} \rightarrow \cal{G}_{V_i}$.

Then $\bar D$ can be blown up via the given family of local Dehn twists $D_{V_i}$
to give 
a
graph-of-groups isomorphism
$D: \cal{G} \rightarrow \cal{G}$ 
if and only if 
every edge $E_i$ of $\bar{\cal{G}}$ 
with terminal endpoint in one of the $V_i$ 
is locally zero. 

\smallskip
\noindent
(2)
The  blow-up 
automorphism
$D: \cal{G} \rightarrow \cal{G}$
obtained in (1) is 
a Dehn twist, and hence $\hat D = \hat{\bar D}$ is 
a Dehn twist automorphism.
\end{cor}

It turns out that the last conclusion of the above corollary is more subtle than it may appear at first sight. In order to explain this, we first consider the following two examples:

\begin{example} 
\label{no-blow-up}
(1) We consider a graph-of-groups isomorphism $\bar{H}: \bar{\cal{G}} \rightarrow \bar{\cal{G}}$ defined as the follows:

\begin{enumerate}
\item[(a)] The graph $\Gamma(\bar{\cal{G}})$ underlying $\bar{\cal{G}}$ consists of a single edge $E$  and two distinct vertices $V = \tau(\bar{E}) \neq V_1 = \tau(E)$. The graph-of-groups $\bar{\cal{G}}$ has 
trivial edge group $G_{E}$, hence trivial edge homomorphisms,
and vertex groups $G_{V} = \< a, b \>$,
$G_{V_1} = \< c \>$.

\item[(b)] The isomorphism $\bar{H}$ acts as the identiy on $\Gamma(\bar{\cal{G}})$ and induces trivial group automorphisms on $G_E$ and on $G_{V_1}$, while the local group automorphism $\bar{H}_{V} : G_{V} \rightarrow G_{V}$ is a Dehn twist automorphism which acts on the generators by the map $a\mapsto a$ and $b \mapsto ba$. 
The correction terms are $\delta(E)=1_{G_{V_1}}$ and $\delta(\bar{E})=aba^{-1}b^{-1}$.
In particular, $\bar{H}$ is a partial Dehn twist relative to the vertex $V$.
\end{enumerate}

\noindent
(2) We now consider a local graph-of-groups isomorphism $H_{V} : \cal{G}_{V} \rightarrow \cal{G}_{V}$ which induces the same Dehn twist automorphism as $\bar{H}_V$. The isomorphism 
$H_V$ 
is defined by the following data:

\begin{enumerate}
\item[(a)] The graph-of-groups $\cal{G}_V$ consists of a single vertex $v$  with $G_v = \< x \>$ and a loop edge $e$ with trivial edge group. The isomorphism $H_V$ is a Dehn twist, in that it acts trivially on the underlying graph $\Gamma(\cal{G}_V)$, the edge group $G_{e}$ and vertex group $G_v$. We choose the correction terms to be $\delta(e)=x$ and $\delta(\bar{e})= 1_{G_v}$. 

\item[(b)] Then $H_{V}$ induces on its fundamental group $\pi_1(\cal{G}_{V}) \cong \< x, t_{e}\>$ an automorphism which sends $x \mapsto x$ and $t_e \mapsto t_e x$. This is exactly the same automorphism as $\bar{H}_{V}$, modulo the identification map $\theta$ given by $a \mapsto x$ and $b \mapsto t_e$.
\end{enumerate}

\noindent
However, in this example $\bar{H}$ cannot be blown up via $H_V$, since the correction term $\delta(\bar{E})$ is mapped by $\theta$ to $xt_{e} x^{-1} t_{e}^{-1}$ which is not $H_{V}^{-1}$-conjugate to any element that has $G_{V}$-length equal to zero: the edge $\bar{E}$ is not locally-zero. 

\end{example}

\begin{example} 
\label{blow-up-possible}
Let $(\bar{H}, \bar{\cal{G}})$ be the partial Dehn twist defined in Example \ref{no-blow-up}.  Instead of $(H_V, \cal{G}_V)$ we now consider a local Dehn twist $(H'_{V}, \cal{G}^{\prime}_{V})$ where $\cal{G}^{\prime}_V$ consists of a single vertex $v'$  with $G_v' = \< x, y \>$, a loop edge $e'$ with cyclic edge group $G_{e'} = \< z \>$, and edge homomorphisms that maps $z$ to $f_{\bar{e}'}(z) = y$ and $f_{e'}(z) = x$.
The correction terms are $\delta(e')=x$ and $\delta(\bar{e}')= 1_{G_{v'}}$. 

Then the fundamental group of $\cal{G}^{\prime}_{V}$ is $\pi_1(\cal{G}_{V}^{\prime}) \cong \< x, y, t_{e'}\ | y= t_{e'} x t_{e'}^{-1} \> \cong \< x, t_{e'} \>$, and $H'_{V}$ induces $\bar{H}_{V}$ via the identification $\theta^{\prime}: a \mapsto x; b \mapsto t_{e'}$.

Contrary to the previous example, one can indeed blow up $(\bar{H}, \bar{\cal{G}})$ via $(H'_{V}, \cal{G}^{\prime}_{V})$ since the identification $\theta^{\prime}$ maps $\delta(\bar{E})$ to $x t_{e'} x^{-1} t_{e'}^{-1} = x y^{-1} \in G_{v'}$: in this example the edge $\bar{E}$ is locally-zero.

\end{example}

It is easy to see that both examples represent the same outer automorphisms of $F_3$.
This shows that the last conclusion of Corollary \ref{intro2}, namely that the given partial Dehn twist $\bar D$ induces on $\pi_1 \bar{\cal G}$ a Dehn twist automorphism $\hat{\bar D}$, is 
not equivalent to the fact 
that the blow-up automorphism exists and is a Dehn twist. However, 
using the terminology of \cite{CL99}, 
it is shown in \cite{KY02} that,
when all the local Dehn twists 
$D_{V_i}$ 
are 
``efficient'' 
(as is the case in Example \ref{blow-up-possible}), then the condition that all 
edges $E_i$
are locally zero is not just sufficient but also necessary for the 
last 
conclusion of Corollary \ref{intro2}. This 
is
used in \cite{KY03} as 
an important 
ingredient of the proof that every linearly growing outer automorphism of a finitely generated free group $F_{n}$ is (up to taking powers) a Dehn twist automorphism.

\smallskip

Since this paper has been made public first, its results have already been used crucially in two applications:
\begin{enumerate}
\item
Corollary \ref{intro2} is a vital ingredient in our algorithmic solution in \cite{KY03} to the question which polynomially growing automorphisms of $F_{n}$ are 
(up to passing to a power)
induced by a surface homeomorphism.
\item
In \cite{LY1} a normal form (based on graph-of-groups and Dehn twists) for quadratically growing automorphisms is given, together with 
a method how to derive 
this normal form. One of the crucial steps in this procedure is to quotient subgraphs and to blow up vertices relative to local Dehn twists, as studied here.
\end{enumerate}

\subsection*{\it Acknowledgements}

${}^{}$

I would like to thank my thesis advisors, Arnoud Hilion and Martin Lustig, for all their guidance and encouragement on this paper, and for all the helpful discussions and suggestions.

\section{Basics of Graphs of groups and their Isomorphisms}

In this section 
we recall some basic knowledge about graph-of-groups as well as their isomorphisms.
Most of our notations are taken from \cite{CL99}; we refer the readers to \cite{Serre80}, \cite{RW15} and \cite{Bass93} for more detailed 
information 
and discussions.

\subsection{Basic Conventions}

${}^{}$

Unless otherwise stated, a {\it graph} refers to a finite, non-empty, connected graph in the sense of Serre (cf. \cite{Serre80}). 

We recall the notations here. For a graph $\Gamma$, we denote by $V(\Gamma)$, $E(\Gamma)$ its {\it vertex set} and {\it edge set} respectively. 
An edge $e \in E(\Gamma)$ is oriented, and we denote by $\bar{e}$ the edge with inverse orientation,
$\tau(e)$ its {\it terminal vertex} and 
$\tau(\bar{e})=\iota(e)$ 
its {\it initial vertex}.

Notice in particular 
that our graph $\Gamma$ is non-oriented. An {\it orientation} of $\Gamma$ refers to a subset $E^{+}(\Gamma) \subset E(\Gamma)$ such that $E^{+}(\Gamma) 
\cup
\bar{E}^{+}(\Gamma)=E(\Gamma)$ and $E^{+}(\Gamma) \cap \bar{E}^{+}(\Gamma)=\emptyset$, where $\bar{E}^{+}(\Gamma)=\{\bar{e} \mid e \in E^{+}(\Gamma)\}$.

For an arbitrary group $G$, we denote by $ad_{x}: G \rightarrow G$ the inner automorphism defined by element  $x \in G$, namely $ad_{x}(g)=xgx^{-1}$ for all $g \in G$.

\subsection{Graphs of Groups}

\begin{defn}\label{graph of groups}
A {\it graph-of-groups} $\cal{G}$ is defined by
$$ \cal{G}=(\Gamma, (G_{v})_{v\in V(\Gamma)}, (G_{e})_{e \in E(\Gamma)}, (f_{e})_{e\in E(\Gamma)})$$
where:
\begin{enumerate}
\item $\Gamma$ is a graph, called the {\it underlying} graph;
\item each $G_{v}$ is a group, called the {\it vertex group} of $v$;
\item each $G_{e}$ is a group, called the {\it edge group} of $e$, and we require $G_{e}=G_{\bar{e}}$ for every $e \in E(\Gamma)$;
\item for each $e\in E(\Gamma)$,
the 
map 
$f_{e}: G_{e} \rightarrow G_{\tau(e)}$ is an injective {\it edge homomorphism}.
\end{enumerate}
\end{defn}

For a graph-of-groups $\cal{G}$, we usually denote by $\Gamma(\cal{G})$ the graph underlying it. The vertex set of $\Gamma(\cal{G})$ is denoted by $V(\cal{G})$ while the edge set is denoted by $E(\cal{G})$.

\begin{defn} \label{bassgroup}
For a graph-of-groups $\cal{G}$, its {\it word group} $W(\cal{G})$ 
is the free product of 
all
vertex groups and 
of
the free group generated by 
the {\it stable letter} $t_{e}$ for every $e \in E(\Gamma)$, i.e. 
$$W(\cal{G})=\underset{v \in V(\Gamma)}{\operatorname{\ast}}G_{v}*F(\{t_{e} \mid e\in E(\Gamma)\}) \, . $$

The {\it path group} (sometimes also called {\it Bass group}) of $\cal{G}$ is defined by 
$$\Pi(\cal{G})= W(\cal{G})/ R \, ,$$
where $R$ is the normal subgroup determined by the following relations:
\begin{enumerate}
\item[$\diamond$] $t_{e}=t_{\bar{e}}^{-1}$, for every $e \in E(\Gamma)$;
\item[$\diamond$] 
$f_{\bar{e}}(g)=t_{e}f_{e}(g)t_{e}^{-1}$,
for every $e \in E(\Gamma)$ and every $g\in G_{e}$.
\end{enumerate}
\end{defn}

\begin{rem}
A {\it word} $w \in W(\cal{G})$ can always be written in the form 
$$w = r_{0}t_{1}r_{1}...r_{q-1}t_{q}r_{q}\ \ \ \  (q \geq 0) \, , $$
where each 
$t_{i} \in \{t_{e} \mid e \in E(\Gamma)\}$
and each 
$r_{i} \in\underset{v \in V(\Gamma)}{\operatorname{\ast}}G_{v}$.

The sequence $(t_{1}, t_{2}, ..., t_{q})$ is called the {\it path type} of $w$, the number $q$ is called the {\it path length},
or sometimes the {\it $\cal{G}$-length} 
of $w$,
denoted by $| w |_{\cal{G}}=q$. In this case, we say that $e_{1}e_{2}...e_{q}$ is the path underlying $w$. 
Two path types $(t_{1}, t_{2}, ..., t_{q})$ and $(t^{\prime}_{1}, t^{\prime}_{2}, ..., t^{\prime}_{s} )$ are said to be same if and only if $q=s$ and $t_{i}=t^{\prime}_{i}$ for each $1 \leq i \leq q$.
\end{rem}

\begin{defn}
Let $w \in W(\cal{G})$ be a word of the form $w = r_{0}t_{1}r_{1}...r_{q-1}t_{q}r_{q}$. The word $w$ is said to be {\it connected} if 
$r_{0} \in G_{\tau(\bar{e}_{1})}$,
$r_{q} \in G_{\tau(e_{q})}$, and 
$\tau(e_{i})=\tau(\bar{e}_{i+1})$
with
$r_{i} \in G_{\tau(e_{i})}$, for $i=1,2,...,q-1$.
We sometimes call such $w$ {\it a connecting word from $\tau(\bar{e}_1)$ to $\tau(e_q)$}, or {\it a word from $\tau(\bar{e}_1)$ to $\tau(e_q)$}.

Moreover, if $w$ is connected and $\tau(e_{q})=\tau(\bar{e}_{1})$, we say that $w$ is a {\it closed connected word issued at the vertex $\tau(e_{q})$}.
\end{defn}

\begin{defn}
Let $w=r_{0}t_{1}r_{1}...r_{q-1}t_{q}r_{q} \in W(\cal{G})$, $w$ is said to be {\it reduced} if it satisfies:
\begin{enumerate}
\item[$\diamond$] if $q=0$, then $w = r_{0}$  
isn't 
equal to the unit element;
\item[$\diamond$] if $q>0$, then whenever $t_{i} = t_{i+1}^{-1}$ for some $1 \leq i \leq q-1$ we have $r_{i}\not\in f_{e_{i}}(G_{e_{i}})$.
\end{enumerate}

Moreover the word $w$ is said to be {\it cyclically reduced} if it is reduced and
if $q>0$ and $t_{1} = t_{q}^{-1}$, then $r_{q}r_{0}\not\in f_{e_{q}}(G_{e_{q}})$.
\end{defn}

We recall the following 
well known 
facts.
\begin{prop}\label{propbassgroup}
For any graph-of-groups $\cal{G}$, the following holds:
\begin{enumerate}
\item Every non-trivial 
element of $\Pi(\cal{G})$ can be represented as a reduced word.
\item Every reduced word is a non-trivial element in $\Pi(\cal{G})$.
\item If $w_{1},w_{2} \in W(\cal{G})$ are two reduced words representing the same element in $\Pi(\cal{G})$, then $w_{1}$ and $w_{2}$ are of the same path type. In particular, $w_{2}$ is connected if and only if $w_{1}$ is connected.

In fact, suppose $w_{1}=r_{0}t_{1}r_{1}...r_{q-1}t_{q}r_{q}$ and $w_{2}=r^{\prime}_{0}t_{1}r^{\prime}_{1}...r^{\prime}_{q-1}t_{q}r^{\prime}_{q}$, then there exist elements $h_{i} \in G_{e_{i}}$ $(i=1,2,...,q)$ such that:

$r^{\prime}_{0}=r_{0}f_{\bar{e}_{1}}(h_{1})$;
$r^{\prime}_{i}=f_{e_{i}}(h_{i})r_{i}f_{\bar{e}_{i+1}}(h^{-1}_{i+1})$ for $(i=1,2,...,q-1)$;
and $r^{\prime}_{q}=f_{e_{q}}(h_{q})r_{q}$.
\end{enumerate}
\qed
\end{prop}

\begin{defn}
For any 
$v_{0} \in V(\Gamma)$ the {\it fundamental group based at $v_{0}$}, denoted by $\pi_{1}(\cal{G}, v_{0})$, consists of the elements in $\Pi(\cal{G})$ that are 
represented by
closed connected words issued at $v_{0}$.

For a vertex $w_{0} \in V(\Gamma)$ different from $v_{0}$, we have $\pi_{1}(\cal{G},v_{0}) \cong \pi_{1}(\cal{G},w_{0})$. In fact, let $W \in \Pi(\cal{G})$ be 
represented by a 
a connected word with underlying path 
from $v_{0}$ to $w_{0}$. The restriction of 
$ad_{W} : \Pi(\cal{G}) \rightarrow \Pi(\cal{G})$ to $\pi_{1}(\cal{G},w_{0})$ induces an isomorphism from $\pi_{1}(\cal{G},w_{0})$ to $\pi_{1}(\cal{G},v_{0})$. Sometimes we write $\pi_{1}(\cal{G})$ when the choice of basepoint doesn't make a difference.
\end{defn}

\subsection{Graph-of-Groups Isomorphism}

${}^{}$

\begin{defn}\label{graphofgroupsiso}
Let $\cal{G}_{1}$, $\cal{G}_{2}$ be two graphs of groups. Denote 
$\Gamma_{1}=\Gamma(\cal{G}_{1})$ and $\Gamma_{2}=\Gamma(\cal{G}_{2})$.
An isomorphism $H: \cal{G}_{1} \rightarrow \cal{G}_{2}$ is a tuple of the form
$$H=(H_{\Gamma}, (H_{v})_{v \in V(\Gamma_{1})}, (H_{e})_{e \in E(\Gamma_{1})}, (\delta(e))_{e \in E(\Gamma_{1})}) \, ,$$
where
\begin{enumerate}
\item $H_{\Gamma}: \Gamma_{1} \rightarrow \Gamma_{2}$ is a graph isomorphism;
\item $H_{v}: G_{v} \rightarrow G_{H_{\Gamma}(v)}$ is a group isomorphism, for any $v \in V(\Gamma_{1})$;
\item $H_{e}=H_{\bar{e}}: G_{e} \rightarrow G_{H_{\Gamma}(e)}$ is a group isomorphism, for any $e \in E(\Gamma_{1})$;
\item for every $e \in E(\Gamma_{1})$, the {\it correction term} 
$\delta(e) \in G_{\tau(H_{\Gamma}(e))}$
is an element such that
    $$H_{\tau(e)}f_{e}=ad_{\delta(e)}f_{H_{\Gamma}(e)}H_{e}.$$
\end{enumerate}
\end{defn}

\begin{rem}
A graph-of-groups isomorphism $H: \cal{G}_{1} \rightarrow \cal{G}_{2}$ induces an isomorphism $H_{*}: \Pi(\cal{G}_{1}) \rightarrow \Pi(\cal{G}_{2})$ defined on the generators by:
\begin{enumerate}
\item[] $H_{*}(g) = H_{v}(g)$, for $g \in G_{v}$, $v \in V(\Gamma_{1})$;
\item[] $H_{*}(t_{e}) = \delta(\bar{e}) t_{H_{\Gamma}(e)} \delta(e)^{-1}$, for $e \in E(\Gamma_{1})$.
\end{enumerate}

It is 
easy to verify 
that $H_{*}$ preserves the relations $t_{e}t_{\bar{e}}=1$ for any $e \in E(\cal{G})$ and $f_{\bar{e}}(g)=t_{e}f_{e}(g)t_{e}^{-1}$, for any $e\in E(\cal{G})$ and $g \in G_{e}$.

Furthermore, the restriction of $H_{*}$ 
to
$\pi_{1}(\cal{G}_{1},v)$, where $v \in V(\Gamma_{1})$, 
defines an isomorphism
$H_{* v}: \pi_{1}(\cal{G}_{1},v) \rightarrow \pi_{1}(\cal{G}_{2}, H_{\Gamma}(v))$.

As in \cite{CL99}, we define the {\it outer isomorphism} induced by a group isomorphism $f: G_{1} \rightarrow G_{2}$ as the equivalence class 
$$\hat{f}=\{ad_{g}f : G_{1} \rightarrow G_{2} \mid g \in G_{2}\}.$$
Hence $H_{*v}$ induces an outer isomorphism 
$\hat{H}_{*v}: \pi_{1}(\cal{G}_{1},v) \rightarrow \pi_{1}(\cal{G}_2,H_{\Gamma}(v))$.

Observe that when choosing a different vertex $v_{1}$ as basepoint, we may choose a word $W \in \Pi(\cal{G}_{1})$ 
with underlying 
path from $v_{1}$ to $v$ 
to obtain the following commutative diagram:
\[
\begin{CD}
\pi_{1}(\cal{G}_{1},v) @>H_{*v}>> \pi_{1}(\cal{G}_{2}, H_{\Gamma}(v)) \\
@Vad_{W} VV @Vad_{H_{*}(W)} VV \\
\pi_{1}(\cal{G}_{1},v_{1})@>>H_{*v_{1}}> \pi_{1}(\cal{G}_{2}, H_{\Gamma}(v_{1}))
\end{CD}
\]

By Lemma 2.2 and Lemma 3.10 in \cite{CL99}, 
the map 
$\hat{H}_{*v}$ determines an outer isomorphism $\hat{H}_{*v_{1}}: \pi_{1}(\cal{G}_{1}, v_{1}) \rightarrow \pi_{1}(\cal{G}_{2}, H_{\Gamma}(v_{1}))$
which is independent of the choice of $W$.
Hence the isomorphism $H: \cal{G}_{1} \rightarrow \cal{G}_{2}$ induces 
a
well-defined
outer isomorphism $\hat{H}: \pi_{1}(\cal{G}_{1}) \rightarrow \pi_{1}(\cal{G}_{2})$ which doesn't depend on the choice of basepoint.
\end{rem}

\begin{rem}[Composition, Inverse]
For two graph-of-groups isomorphisms $H^{\prime}: \cal{G}_{1} \rightarrow \cal{G}_{2}$, $H^{\prime\prime}: \cal{G}_{2} \rightarrow \cal{G}_{3}$, the {\it composition} of $H^{\prime}$ and $H^{\prime\prime}$ is an isomorphism $H^{\prime\prime}H^{\prime}=H: \cal{G}_{1} \rightarrow \cal{G}_{3}$ 
given (for any $v \in V(\Gamma_{1})$, $e \in E(\Gamma_{1})$) precisely by:
$H_{\Gamma}=H_{\Gamma}^{\prime\prime}H_{\Gamma}^{\prime}$;
$H_{v}= H_{H^{\prime}_{\Gamma}(v)}^{\prime\prime}H^{\prime}_{v}$;
$H_{e}= H_{H^{\prime}_{\Gamma}(e)}^{\prime\prime}H^{\prime}_{e}$;
$\delta(e)=H_{\tau(H^{\prime}_{\Gamma}(e))}^{\prime\prime}(\delta^{\prime}(e))\delta^{\prime\prime}(H^{\prime}_{\Gamma}(e))$.
Moreover, $H$ satisfies 
$H_{*}=H^{\prime\prime}_{*}H^{\prime}_{*}$ and $\hat{H}=\hat{H}^{\prime\prime}\hat{H}^{\prime}$.

For 
any 
graph-of-groups isomorphism $H: \cal{G}_{1} \rightarrow \cal{G}_{2}$ 
the 
{\it inverse} isomorphism is $H^{-1}: \cal{G}_{2} \rightarrow \cal{G}_{1}$, which satisfies $H_{*}^{-1}=(H^{-1})_{*}$ and $\hat{H}^{-1}=\hat{H^{-1}}$, is defined 
(for all $v \in V(\Gamma_{2})$, $e \in E(\Gamma_{2})$)
by:
$ (H^{-1})_{\Gamma}=(H_{\Gamma})^{-1}$;
$ (H^{-1})_{v}=(H_{H_{\Gamma}^{-1}(v)})^{-1}$;
$ (H^{-1})_{e}=(H_{H_{\Gamma}^{-1}(e)})^{-1}$;
$ \delta^{-1}(e)= H^{-1}_{H^{-1}(\tau(e))}(\delta(H^{-1}_{\Gamma}(e))^{-1})$.

\end{rem}

\subsection{A Natural Equivalence Between Graphs Of Groups}\label{natural equivalence}

${}^{}$

Suppose 
that 
$\cal{G}$, $\cal{G}^{\prime}$ are two graphs-of-groups, 
and that
$\cal{G}^{\prime}$ equals to $\cal{G}$ everywhere except that for 
some 
edge $e_{0} \in E(\cal{G})$
one has 
$f^{\prime}_{e_{0}}=ad_{w_{e_{0}}^{-1}} \circ f_{e_{0}}$, where $w_{e_{0}}$ is an element 
in $G_{\tau(e_{0})}$. 
Then there is a natural isomorphism between $\cal{G}$ and $\cal{G}^{\prime}$.

More concretely, define $H_{0}: \cal{G} \rightarrow \cal{G}^{\prime}$ by the rules:

\begin{enumerate}
\item $H_{0, \Gamma}=id_{\Gamma(\cal{G})}$; 
\item $H_{0,v}=id_{G_{v}}$ for 
any 
$v \in V(\cal{G})$; $H_{0,e}=id_{G_{e}}$ for 
any 
$e \in E(\cal{G})$;
\item 
$\delta_{0}(e_0)=w_{e_{0}}$, and 
$\delta_{0}(e)= 1$ 
when $e \neq e_0$.
\end{enumerate}

Then it's  easy to verify that $H_{0}$ is a 
well-defined
graph-of-groups isomorphism, since the additional compatibility requirement holds automatically for all edges $e_{0} \neq e \in \Gamma_{0}$, and for $e_{0}$ we have
\begin{align*}
H_{0,\tau(e_0)} \circ f_{e_0} 
= f_{e_0} = ad_{w_{e_0}} \circ f^{\prime}_{e_0} 
= ad_{\delta_{0}(e_0)}\circ f^{\prime}_{e_0}
= ad_{\delta_{0}(e_0)}\circ f^{\prime}_{e_0} \circ H_{0,e_0 \, }.
\end{align*}

The above isomorphism gives rise to a natural notion of ``equivalent'' graphs-of-groups, where the equivalence relation is generated by isomorphisms of the above type $H_0$ as elementary equivalence.  This notion of ``equivalent'' graph-of-groups, although not really established in the literature, is natural, in that it preserves (up to canonical isomorphisms) the fundamental group. It also shows up in the prime feature of graph-of-groups, meaning Bass-Serre theory:  Given a group $G$ that acts on a (simplicial) tree $T$, for the associated graph-of-groups decomposition $G \cong \pi_1 (\cal G_T)$ the ``quotient'' graph-of-groups $\cal G_T$ of $T$ modulo $G$ is only well defined up to precisely this equivalence relation.

\begin{lem}\label{basiclemma}
Let $\cal{G}$, $\cal{G}^{\prime}$, $H_0$,
$e_0$ and $w_0$ 
be 
as 
defined as above.

Let $H: \cal{G} \rightarrow \cal{G}$ be 
a graph-of-groups 
automorphism,  
and let 
$H'=(H'_{\Gamma}$, $(H'_{v})_{v \in V(\cal G')}$, $(H'_{e})_{e \in E(\cal G')}$, $(\delta'(e))_{e \in E(\cal G')})$
be equal to $H$ everywhere except that $\delta^{\prime}(e_{0})=$ $H_{\tau(e_{0})}(w_{e_{0}})^{-1} \delta(e_{0}) w_{e_{0}}$\,.

Then $H^{\prime}$ 
determines 
a 
well-defined
graph-of-groups automorphism which is conjugate to $H$ 
via 
$H_{0}$.
More precisely, we have $H^{\prime}=H_{0} \circ H \circ H_{0}^{-1}$, 
and hence in particular $\hat{H}^{\prime}= \hat{H}_{0} \circ \hat{H} \circ \hat{H}_{0}^{-1}$.
\end{lem}

\begin{proof}
In order
to 
show that 
$H^{\prime}$ is a 
well-defined
graph-of-groups isomorphism, 
it is 
sufficient to verify that $H_{\tau(e_0)} \circ f^{\prime}_{e_0}=ad_{\delta^{\prime}(e_{0})} \circ f^{\prime}_{e_0} \circ H_{e_0}$.

For every $g \in G_{e_0}$
we have
\begin{align*}
ad_{\delta^{\prime}(e_{0})} \circ f^{\prime}_{e_0} \circ H_{e_0} (g)
&= H_{\tau(e_{0})}(w_{e_{0}})^{-1} \delta(e_{0}) w_{e_0} f^{\prime}_{e_0}(H_{e_0}(g)) w_{e_0}^{-1} \delta(e_0)^{-1}H_{\tau(e_{0})}(w_{e_{0}})\\
&= H_{\tau(e_{0})}(w_{e_{0}})^{-1} \delta(e_{0}) f_{e_0}(H_{e_0}(g))\delta(e_0)^{-1} H_{\tau(e_{0})}(w_{e_{0}})\\
&= H_{\tau(e_{0})}(w_{e_{0}})^{-1} H_{\tau(e_0)}(f_{e_0}(g)) H_{\tau(e_{0})}(w_{e_{0}}) \\
&= H_{\tau(e_0)}(w_{e_0}^{-1}f_{e_0}(g)w_{e_0})\\
&= H_{\tau(e_0)}\circ f^{\prime}_{e_0}(g)
\end{align*}

Moreover, we have the following diagram commutes:
\[
\begin{CD}
\cal{G} @>H>> \cal{G} \\
@VH_{0} VV @VVH_{0} V \\
\cal{G}^{\prime}  @>>H^{\prime}> \cal{G}^{\prime} 
\end{CD}
\]

The only non-trivial part 
here 
is to verify that the following equation holds for all edges:
\begin{align*}
H_{0,H_{\Gamma_0}(\tau(e))}(\delta(e)) \delta_0(H_{\Gamma}(e))
= H^{\prime}_{H_{\Gamma_0}(\tau(e))}(\delta_{0}(e))\delta^{\prime}(H_{\Gamma_0}(e))
\end{align*}

When $e \neq e_0$, this equation 
follows from
$\delta(e)=\delta^{\prime}(e)$ which holds automatically by definition.

When $e = e_0$, 
we 
Compute both, the left and the right hand side:
\begin{align*}
\text{Left:}   
\,\,\, H_{0,H_{\Gamma_0}(\tau(e))}(\delta(e)) \delta_0(H_{\Gamma}(e)) 
&
=\delta(e_0)\delta_0(e_0) = \delta(e_0) w_{e_0} \\
\text{Right:} 
\,\,\, H^{\prime}_{H_{\Gamma_0}(\tau(e))}(\delta_{0}(e))\delta^{\prime}(H_{\Gamma_0}(e))
&
=H^{\prime}_{\tau(e_0)}(\delta_0(e_0))\delta^{\prime}(e_0) \\
& 
=H^{\prime}_{\tau(e_0)}(w_{e_0})\delta^{\prime}(e_0) \\
&
=H_{\tau(e_0)}(w_{e_0})\delta^{\prime}(e_0)
=\delta(e_0) w_{e_0}
\end{align*}

Hence we have 
shown that 
the equation holds for all edges, and that $H^{\prime}=H_{0} \circ H \circ H_{0}^{-1}$, which implies $\hat{H}^{\prime}= \hat{H}_{0} \circ \hat{H} \circ \hat{H}_{0}^{-1}$.
\end{proof}

\begin{rem}
\label{Bass-Serre-gog-auto}
Similar to the equivalence between $\cal G$ and $\cal G'$ discussed above, Lemma \ref{basiclemma} gives rise to an equivalence between graph-of-groups automorphisms, which is natural in the following sense: 

Assume that 
some group $G$ acts on a simplicial tree $T$, 
and let $\cal G_T$ be the ``quotient graph-of-groups''  of $T$ modulo $G$ mentioned above.
Assume furthermore that for some outer automorphism $\phi$ of $G$ 
the tree 
$T$ is ``$\phi$-invariant'', by which one means that the translation length function $\| \cdot \|_{T}$, defined on the conjugacy classes of $G$ by setting all edge lengths equal to 1, is preserved by $\phi$:
$$\|\phi [g]\|_{T} = \|[g]\|_{T} \quad \text{for all} \quad g \in G$$
Now Bass-Serre theory is set up in such a way that this $\phi$-invariance of $T$ is equivalent to the existence of
a graph-of-groups automorphism $H: \cal G_T \to \cal G_T$ which induces the given automorphism $\phi$.

However, in this situation the automorphism $H$ is determined only up to 
an equivalence which is generated by the elementary equivalence  $H \sim H'$, where $H$ and $H'$ are precisely as given in Lemma \ref{basiclemma} above.
\end{rem}

A statement analogous to Lemma~\ref{basiclemma} with $e_0$ replaced by a family of edges can be derived by applying Lemma~\ref{basiclemma} iteratively.

\section{Dehn Twists}

\subsection{Classical Dehn Twists}

${}^{}$

We first recall the classical definition of 
a 
Dehn twist 
as given in \cite{CL95}. 

\begin{defn}[Classical Dehn twist]
\label{Dehn-twist-defn}
An automorphism 
$D: \cal{G} \rightarrow \cal{G}$
of a graph-of-groups $\cal{G}$ is called a {\it 
(classical) 
Dehn twist} if it satisfies:

\begin{enumerate}
\item $D_{\Gamma} = id_{\Gamma}$;
\item $D_{v} = id_{G_{v}}$, for all $v \in V(\Gamma)$;
\item $D_{e} = id_{G_{e}}$, for all $e \in E(\Gamma)$;
\item for each $G_{e}$, there is an element $\gamma_{e} \in Z(G_{e})$ such that $\delta(e)=f_{e}(\gamma_{e})$, where $Z(G_{e})$ denotes the center of $G_{e}$.
\end{enumerate}

We denote a Dehn twist defined as above by 
$D=D(\cal{G},(\gamma_{e})_{e \in E(\cal{G})})$.
\end{defn}

\begin{defn}[Twistor] 
Given a Dehn twist 
$D=D(\cal{G},(\gamma_{e})_{e \in E(\cal{G})})$, 
we define the 
{\it twistor} of an edge $e\in E(\Gamma)$ 
by 
setting
$z_{e}=\gamma_{\bar{e}}\gamma_{e}^{-1}$.
Then 
for any edge $e$ 
we have 
$z_{e} \in Z(G_{e})$ and $z_{\bar{e}}=\gamma_{e} \gamma_{\bar{e}}^{-1}=z^{-1}_{e}$.
\end{defn}

\begin{rem}
The induced automorphism $D_{*} :\Pi(\cal{G}) \rightarrow \Pi(\cal{G})$ 
is defined on generators as
follows:
\begin{enumerate}
\item[] $D_{*}(g)=g$, for $g \in G_{v}$, $v\in V(\Gamma)$;
\item[] $D_{*}(t_{e}) = t_{e}f_{e}(z_{e})$, for every $e \in E(\Gamma)$.
\end{enumerate}

In particular, the induced automorphism on the fundamental group, $D_{*v}: \pi_{1}(\cal{G}, v) \rightarrow \pi_{1}(\cal{G},v)$, where $v \in V(\Gamma)$, is called a 
{\it Dehn twist automorphism}.
\end{rem}

\begin{rem}
It follows from Proposition 5.4 in \cite{CL99}
that in many situations a Dehn twist on a given graph-of-groups is uniquely determined by its twistors.
Thus sometimes we may define a Dehn twist by its twistors $(z_{e})_{e\in E(\Gamma)}$ 
(With $z_{e} \in Z(G_{e})$ and $z_{\bar{e}}=z_{e}^{-1}$, for each $e\in E(\Gamma)$). 
In this case, we may conversely define:
\begin{equation*}
\gamma_{e}=\left\{
\begin{aligned}
& z_{e}^{-1}, & e\in E^{+}(\Gamma) \\
& 1,  & e\in E^{-}(\Gamma).
\end{aligned}
\right.
\end{equation*}
\end{rem}

\begin{defn}\label{Dehntwist}
A group automorphism $\phi : G \rightarrow G$ is said to be a {\it Dehn twist automorphism} if 
it is 
represented by a graph-of-groups Dehn twist. In other words, there exists a graph-of-groups
$\cal G$, a Dehn twist $D: \cal{G} \rightarrow \cal{G}$, 
and an 
isomorphism
$\theta: G \rightarrow \pi_{1}(\cal{G},v)$ such that $\phi=\theta^{-1} \circ D_{*v} \circ \theta$.

In this case the induced outer automorphism 
$\hat{\phi} \in Out(G)$
is called a {\it Dehn twist outer automorphism}.
\end{defn} 

\begin{rem}
Notice that,
for a 
given
Dehn twist automorphism $\phi: G \rightarrow G$, its Dehn twist representative is 
in general
not unique. 

In the special case where 
$G$ is a free group, 
such ``unique'' Dehn twist representatives are given by {\em efficient} Dehn twists. For details see \cite{CL95} or Section~3.3 in \cite{KY02}.

\end{rem}

\begin{rem}
It is easy to 
see 
that every multiple Dehn twist homeomorphism $h$ on a 
compact 
surface $S$ (possibly with finitely many boundary components), as defined in \cite{FM}, gives rise to a graph-of-groups Dehn twist $D: \cal{G} \rightarrow \cal{G}$. 
Here $\cal G$ is 
a graph-of-groups decomposition of the surface group 
$\pi_{1}(S)$
modeled the decomposition of $S$ by the twistor curves,  
and 
the 
Dehn twist $D$ on 
$\cal{G}$
defines 
the same outer automorphism on 
$\pi_{1}(\cal{G}) \cong \pi_{1}(S)$ 
as the 
given multiple 
Dehn twist homeomorphism $h$. 
See \cite{KY-thesis} and section 6 of \cite{KY03} for a more detailed explanation.
\end{rem}

\subsection{General Dehn Twists}
${}^{}$
There are several places in the literature (e.g. see \cite{ClayPettet10, Levitt, RipsSela94}) where the notion of a ``Dehn twist'' by means of graph-of-groups automorphisms have been defined; these definitions all agree 
in essence,
but are slightly distinct in their technical specifications. 
Thus subsection is meant as
contribution to 
a unification of these concepts.

Among the various alternatives to classical Dehn twists as presented in the previous subsection, the idea to simplify the concept of twists on non-trivial edge group elements through keeping the edge groups trivial and working with ``interesting'' correction terms has its strong merits, but also some defaults (see 
the opening paragraph of subsection \ref{partial-D}).
It leads to graph-of-groups automorphisms as considered in the next proposition; these are particularly important as they occur 
naturally in the context of Bestvina-Handel's train track maps in \cite{BH92}.

\begin{prop}\label{trivial edge groups}
${}$
Let $\cal{G}$ be a graph-of-groups such that 
the edge groups $G_e$ 
are trivial 
for all $e \in E(\cal{G})$. Let $H: \cal{G} \rightarrow \cal{G}$ be an automorphism such that 
\begin{enumerate}
\item[$\diamond$] $H_{\Gamma}$ acts on $\Gamma(\cal{G})$ as identity;
\item[$\diamond$] $H_{v}: G_{v} \rightarrow G_{v}$ is 
the identity map, 
for all $v \in V(\cal{G})$. 
\end{enumerate}
Then the induced automorphism $\hat H \in \Out(\pi_1 \cal G)$ is a Dehn twist automorphism.
\end{prop}

The proof of this proposition is postponed (see 
Remark \ref{three-eleven}); 
we first want to enlarge our concept of a ``Dehn wist automorphism'' to include other notions mentioned above. Our Definition \ref{dehn} below includes in particular also the concept set up in 
\cite{Levitt}, which to our knowledge is the most general among the ones presently in the literature. Be aware, however, that in \cite{Levitt} Levitt's terminology for ``Dehn twist'' corresponds to what is called here ``Dehn twist outer automorphism'' 
(see Definition~\ref{Dehntwist}).

In any case, it will be shown in Proposition \ref{equivalent Dehn twist} below that all these generalizations differ from the original ``classical'' Dehn twist concept only in their presentation and not really in substance.

\begin{defn}[General Dehn twist]
\label{dehn}
Let $\cal{G}$ be a graph-of-groups. An automorphism $D: \cal{G} \rightarrow \cal{G}$ is called a {\it {general Dehn twist}} if 

\begin{enumerate}
\item[$\diamond$] $D_{\Gamma}=id_{\Gamma}$;
\item[$\diamond$] $D_{v}=id_{G_{v}}$, for 
any
vertex $v \in V(\cal{G})$;
\item[$\diamond$] $D_{e}=id_{G_{e}}$, for 
any 
edge $e \in E(\cal{G})$;
\item[$\diamond$] $\delta(e) \in C(f_{e}(G_{e}))$, where $C(f_{e}(G_{e}))$ denotes the centeralizer of $f_{e}(G_{e})$ in $G_{\tau(e)}$, for 
any 
$e \in E(\Gamma)$. 
\end{enumerate}
\end{defn}

\begin{prop}
\label{equivalent Dehn twist}
The 
notions of a
``classical Dehn'' twist and 
of a 
``general Dehn'' twist are equivalent in the following sense:
\begin{enumerate}
\item
Every classical Dehn twist is a general Dehn twist.
\item
To every general Dehn twist $D: \cal G \to \cal G$ we can canonically associate a 
classical
Dehn twist $D': \cal G' \to \cal G'$ and an isomorphism $\theta_v: \pi_1(\cal G, v) \to \pi_1(\cal G', v')$ (for any vertex $v$ 
of 
$\cal G$ and a corresponding vertex $v'$ 
of
$\cal G'$) such that $D'_{* v'} \circ \theta_v = \theta_v \circ D_{* v}$.
\end{enumerate}
In particular, the outer automorphism $\hat{D}$ defined by a general Dehn twist is a Dehn twist automorphism.
\end{prop}

\begin{proof}
(1)
This follows immediately from $f_e(Z(G_e)) \subset C(f_e(G_e))$.

\smallskip
\noindent
(2)
We 
first 
consider a 
special case: 

Let 
$D: \cal{G} \rightarrow \cal{G}$ 
be a general Dehn twist, 
with all correction terms 
as 
defined in the classical 
case 
except for 
a single 
edge 
$e$:
we suppose 
$\delta(e) \in C(f_{e}(G_{e}))\subset G_{\tau(e)}$
but 
$\delta(e) \notin f_{e}(Z(G_e))$.

Then we  
define 
a classical Dehn twist $D^{\prime}: \cal{G}^{\prime} \rightarrow \cal{G}^{\prime}$ 
which is 
obtained 
as follows:
\begin{enumerate}
\item [$\diamond$] 
The 
graph-of-groups $\cal{G}^{\prime}$ is obtained from $\cal G$ 
by 
introducing a new vertex $v_0$
which subdivides 
the 
edge 
$e$
into 
$e^{\prime}$
and
$e^{\prime \prime}$, 
with 
$\iota(e^{\prime})=\iota(e)$, 
$\tau(e^{\prime})=v_{0}=\iota(e^{\prime \prime})$ 
and 
$\tau(e^{\prime \prime})=\tau(e)$, 
by 
setting 
$$G_{e^{\prime}}=G_{e}
\qquad \text{and} \qquad
G_{e^{\prime \prime}}= G_{v_{0}}= \langle f_{e}(G_{e}), \delta(e) \rangle ,$$
and 
by 
defining the edge homomorphisms 
through 
\begin{align*}
& f_{e^{\prime}}(x) = f_{e}(x),\  f_{\bar{e}^{\prime}}(x)=f_{\bar{e}}(x)\  
\quad \text{for all} \quad
x \in G_{e^{\prime}}=G_e \, , \\
&
\quad \text{and} \quad
f_{e^{\prime \prime}}=id_{G_{e''}}, \  f_{\bar{e}^{\prime \prime}}=id_{G_{e''}}. 
\end{align*}

\item [$\diamond$] 
The 
Dehn twist $D^{\prime}$ is defined 
by 
setting 
$$ D^{\prime}_{e^{\prime}}=id_{G_{e^{\prime}}}, \  
D^{\prime}_{e^{\prime \prime}}=id_{G_{e^{\prime \prime}}}, \ 
D^{\prime}_{v_0}=id_{G_{v_{0}}},$$
and 
by 
choosing
$$\delta^{\prime}(\bar{e}^{\prime})=\delta(\bar{e}), \
\delta^{\prime}(e^{\prime})=1_{G_{v_0}},\ 
\delta^{\prime}(\bar{e}^{\prime \prime})=1_{G_{v_0}}, \
\delta^{\prime}(e^{\prime \prime})=\delta(e).$$  
The data for 
$D^{\prime}$  
are 
equal 
to those for 
$D$ everywhere else.

\end{enumerate} 

Since $\delta(e)$ is assumed to lie in  $C(f_{e}(G_{e}))$, any element in $f_{e}(G_{e})$ commutes with $\delta(e)$. As $\delta(e)$ also commutes with itself, we derive immediately that $\delta^{\prime}(e^{\prime \prime})=\delta(e)$ commutes with all elements contained in $\<f_{e}(G_{e}), \delta(e) \> = f_{e''}(G_{e^{\prime\prime}})$, i.e. $\delta^{\prime}(e^{\prime \prime}) \in Z(f_{e''}(G_{e^{\prime\prime}}))=f_{e''}(Z(G_{e^{\prime\prime}}))$. Thus $D^{\prime}$ is a classical Dehn twist as given in Definition \ref{Dehntwist}.

The Dehn twist $D^{\prime}$ is ``equal'' to $D$ in the following sense: for every vertex $v \neq v_{0}$ from $\cal G'$ there is a corresponding vertex for $\cal G$ which we also call $v$.

Consider the homomorphism $\theta: \Pi(\cal{G}) \rightarrow \Pi(\cal{G}^{\prime})$ defined on generators by $t_{e} \mapsto t_{e^{\prime}}t_{e^{\prime \prime}}$ and by $g \mapsto g$ otherwise. 
Then we claim that, for every vertex $v \neq v_0$, the restriction of $\theta$ to $\pi_{1}(\cal{G},v)$ defines an isomorphism $\theta_{v}: \pi_{1}(\cal{G},v) \rightarrow \pi_{1}(\cal{G}^{\prime} , v)$. 

To see this, we first observe that because of 
$G_{v_0}= f_{\bar{e}^{\prime\prime}}(G_{e^{\prime\prime}})$ a reduced word in $\Pi(\cal G')$ can not contain as subword any word of 
type
$t_{\bar e''} g t_{e''}$ with $g \in G_{v_0}$. Since furthermore $e''$ and $\bar e'$ are the only edges issuing from $v_0$, it follows 
for 
any reduced word
$W \in \pi_1(\cal G', v)$ 
with 
$v \neq v_0$,
after appropriately applying the relation $f_{{e'}}(g)=t_{e'}^{-1}f_{\bar e'}(g)t_{e'}$ from Definition \ref{bassgroup}, that any occurrence of $t_{e''}$ in $W$ is preceded by $t_{e'}$ and any $t_{\bar e''}$ is succeeded by $t_{\bar e'}$. This proves the surjectivity of the map $\theta_v$, since by Proposition~\ref{propbassgroup} (1) it suffices to consider reduced words. The injectivity is a direct consequence of part (3) of the same proposition.

To conclude the proof in the special case we now observe that 
$\theta$ 
gives rise to the following 
diagram:
\[
\begin{CD}
\pi_{1}(\cal{G},v) @>D_{*v}>> \pi_{1}(\cal{G},v) \\
@V\theta_{v} VV @VV\theta_{v} V \\
\pi_{1}(\cal{G}^{\prime}, v)  @>>D^{\prime}_{*v}> \pi_{1}(\cal{G}^{\prime}, v) 
\end{CD}
\]
It follows directly from the definition of the maps involved that this diagram is commutative;
the only 
non-trivial 
argument 
needed 
is 
given by:

\begin{align*}
\theta \circ D_{*}(t_{e}) 
 = \theta(\delta(\bar{e})t_{e} \delta(e)^{-1}) 
& = \delta(\bar{e}) t_{e^{\prime}} t_{e^{\prime \prime}} \delta(e)^{-1} \\
& = D_{*}^{\prime}(t_{e^{\prime}})D_{*}^{\prime }(t_{e^{\prime \prime}})
 = D^{\prime}_{*}(t_{e^{\prime}} t_{e^{\prime \prime}}) \\
& = D^{\prime}_{*} \circ \theta(t_{e}).
\end{align*}

In the general case, where $D$ is 
a general Dehn twist which 
may 
have more than one correction 
term 
defined in the 
``non-classical way'', we may apply the above 
treated special case 
repeatedly 
to each of the ``non-classical'' correction terms, 
to eventually obtain a classical Dehn twist.
\end{proof}

\begin{rem}
\label{three-eleven}
One obtains now the statement of 
Proposition~\ref{trivial edge groups} 
as direct consequence of 
Proposition \ref{equivalent Dehn twist} (2), 
since the graph-of-group automorphism $H$ from Proposition \ref{trivial edge groups} is clearly a general Dehn twist.

\end{rem}

\subsection{Partial Dehn twists}
\label{partial-D}

${}^{}$

The type of Dehn twists as considered in Proposition \ref{trivial edge groups} has a strong appeal, due to its simplicity. It is furthermore of natural interest because it is used in relative train track theory (see \cite{BH92}). However, it should be noted that not every outer Dehn twist automorphism, even for a free group of finite rank, can be represented by such simple graph-of-groups automorphisms, as illustrated by the following example:

\begin{rem}
\label{example-Brian}
We consider $F_3 = F(a,b,c)$ and the automorphism $\phi \in \Out(F_3)$ which acts on the generators by sending $a \mapsto a$, $b \mapsto ba$ and 
$c \mapsto (aba^{-1}b^{-1}) c (aba^{-1}b^{-1})^{-1}$.
This is the Dehn twist automorphism induced by the partial Dehn twist given in Example~\ref{no-blow-up}.

In order to see that this Dehn twist automorphism 
can not be realized by a Dehn twist $D: \cal G \to \cal G$ as in Proposition \ref{trivial edge groups}, i.e. based on a graph-of-groups $\cal G$ with all edge groups trivial, we observe that for such a graph-of-groups every vertex group is a free factor of $\pi_1 \cal G \cong F(a, b, c)$. From the algebraic prescriptions 
$a \mapsto a, \, b \mapsto ba, \, c \mapsto (aba^{-1}b^{-1}) \, c \, (aba^{-1}b^{-1})^{-1}$ 
we derive that there must be two edges in $\cal G$, with precisely one of them a loop edge. As a consequence that there must also be precisely two vertices in $\cal G$, and none of the vertex groups can be trivial. It follows that both vertex groups must have rank 1. Since the conjugacy class of the two twisters are prescribed by the action on $\pi_1 \cal G$, it follows that one of the edges at one of its endpoints must have as correction term an element conjugate to $aba^{-1}b^{-1}$. Such a commutator, however, can not be contained in any of the vertex groups, if the latter is a free factor of $F(a, b, c)$ of rank 1. Hence a graph-of-groups $\cal G$ as required for the Dehn twist in question does not exist.
\end{rem}

In order to make up for this defect, we now introduce the following notion of {\it partial Dehn twist}, which is more thoroughly studied in \cite{KY02}, \cite{KY03}, as well as Section~\ref{partialDehntwist} of this article .

\begin{defn}[Partial Dehn twist]
\label{partial Dehn twist}
Let $\cal{G}$ be a graph-of-groups, 
and let $\cal{V}_{0} \subset V(\cal{G})$ be a set of vertices which has the property 
that 
any edge $e$ 
with 
$\tau(e)=v \in \cal{V}_0$
has trivial edge group
$G_e$.

A {\it partial Dehn twist relative to $\cal{V}_{0}$} is a graph-of-groups isomorphism $H:\cal{G} \rightarrow \cal{G}$ such that 
\begin{enumerate}

\item[$\diamond$] $H_{\Gamma}: \Gamma(\cal{G}) \rightarrow \Gamma(\cal{G})$ is 
the 
identity;
\item[$\diamond$] $H_{e}=id_{G_{e}}$ for all edges 
$e \in E(\cal G)$;
\item[$\diamond$] $H_{v}=id_{G_{v}}$ for all vertices $v \notin \cal{V}_0$ 
(while 
$H_{v}$ is 
any 
group isomorphism for all $v \in \cal{V}_{0}$);
\item[$\diamond$] $\delta(e) \in C(f_{e}(G_{e}))$ for all edges
$e \in E(\cal G)$.

\end{enumerate}

\end{defn}

The case which is of most interest to us is that of a partial Dehn twist where all non-trivial vertex group automorphisms are Dehn twist automorphisms themselves. In order to simplify the notation, we include the identity map as {\em trivial Dehn twist} defined by a degenerate graph-of-groups that is based on the graph which consists of a single vertex only. We define:

\begin{defn}
\label{partial-D-rel-D}
${}^{}$
A graph-of-groups automorphism $H: \cal G \to \cal G$ is a {\em partial Dehn twist relative 
to a family of 
Dehn twist automorphisms} if $H$ is a partial Dehn twist as in Definition \ref{partial Dehn twist}, with the specification that
for every vertex $v \in V(\cal G)$ the 
map $H_v$ is a (possibly trivial) Dehn twist automorphism.
\end{defn}

\section{H-conjugation}\label{Hconj}

${}^{}$
Recall from Proposition \ref{propbassgroup} that for any element $W$ in the Bass group $\Pi(\cal G)$, represented by some reduced word $w \in W(\cal G)$, then any other reduced $w' \in W(\cal G)$ which also represents $W$ is connected if and only if $w$ is connected. We hence call $W$ in this case a {\em connected} element of $\Pi(\cal G)$. Similarly, the initial and terminal vertices $\iota(W)$ and $\tau(W)$ are well defined.

\begin{defn}
\label{H-conjugate}
Let 
$H: \cal G \to \cal G$ be an isomorphism of a graph-of-groups.
Let $W_{1}, W_{2}$ be 
non-trivial connected 
elements 
in the Bass group $\Pi(\cal{G})$. 
Then 
$W_{1}$  is said to be {\it {H-conjugate}} to $W_{2}$ if there exists 
a connected 
element 
$W\in \Pi(\cal{G})$ such that $W_{1}=W W_{2} H_{*}(W)^{-1}$.
This connected element $W$ is called {\it $H$-conjugator}.
\end{defn}

\begin{lem}
$H$-conjugacy
is 
an
equivalence relation on 
the set of non-trivial connected 
elements 
in 
$\Pi(\cal{G})$.
\end{lem}

\begin{proof}
Reflexivity and 
symmetry 
are obvious. 
In order to show
transitivity 
we observe that from 
$W_{1}= W W_{2} H_{*}(W)^{-1}$ and $W_{2}= W^{\prime} W_{3} H_{*}(W^{\prime})^{-1}$,
where $W$
and 
$W'$ are connected 
elements, 
one deduces 
$W_{1}= W W^{\prime} W_{3} H_{*}(W W^{\prime})^{-1}$.
Here 
$W W'$ 
is connected since 
$W_1$ and $W_2$ are non-trivial and connected, which implies that $W$ terminates at $\iota(W_2)$ while $W'$ initiates at $\iota(W_2)$.
\end{proof}

\begin{rem}
Let $H$ be an isomorphism of a graph-of-groups $\cal{G}$. 
Two 
non-trivial connected 
elements
$W_1, W_2 \in \Pi(\cal{G})$ are $H$-conjugate to each other if and only if $W^{-1}_1$ and $W^{-1}_2$ are $H^{-1}$-conjugate to each other.
\end{rem}

\begin{defn}
\label{H-trivial}
Let $H: \cal G \to \cal G$ be an isomorphism of a graph-of-groups, and let $W \in \Pi(\cal G)$ be a non-trivial connected 
element.

Then $W$  is said to be {\it {H-zero}} if there exists a connected 
element
$W' \in \Pi(\cal{G})$ such that the (possibly trivial) 
element
$W' W H_{*}(W')^{-1}$ 
is contained in 
a single
vertex group of $\cal G$.
\end{defn}

\begin{rem}
\label{comment-on-zero}
${}^{}$
There is an important particular reason why the trivial element of $\Pi(\cal G)$ is not contained in any of the $H$-conjugacy classes as they are defined in the above set-up. 
It is a rather tricky issue, which has been dealt with in detail in  Section~3 of \cite{LY1}.

This leads, however, to the following ``unexpected'' situation, due to the fact (obtained directly from Definitions \ref{H-conjugate}  and \ref{H-trivial}) that an element $W \in \Pi(\cal G)$ is $H$-zero if and only if any $H$-conjugate $W' \in \Pi(\cal G)$ is also $H$-zero: It could well be that all elements in the $H$-conjugacy class of $W$ are $H$-zero, but none of them is actually contained in some vertex group of $\cal G$, since the only such which is $H$-conjugate to $W$ is the trivial element.

\end{rem}

\begin{rem}
\label{H-zero-algorithmic}
It is also important noting that in the special case, where $\pi_1\cal G$ is a free group of finite rank and $H: \cal G \to \cal G$ is based on the identity map $H_\Gamma = id_{\Gamma(\cal G)}$, one can decide algorithmically whether a non-trivial connected element $W \in \Pi(\cal{G})$ is $H$-zero or not. 

${}^{}$
Indeed, 
by Definition \ref{H-trivial}, 
$W$ is $H$-zero if and only if one can write $W$ as product
$$W = W_1 g H_*(W_1)^{-1} \, ,$$
where $W_1 \in \Pi(\cal G)$ is also connected, and $g$ is contained in some vertex group $G_v$ of $\cal G$. 
By properly chosing $W_1$ we can assume here that $W_1$ is written as reduced word, and that there is no cancellation (other than within $G_v$) in the above product, which is hence reduced and of even $\cal G$-length 
$2r$, for $r = |W_1|_\cal G$. 

Thus, if $W$ is $H$-zero, it can be written as product $W = W' W''$, with $W'$ and $W''$ of $\cal G$-length $r$. The element $W'$ is not quite determined by $W$, but for all possible choices we always have $W'^{-1} W_1 \in G_v$. We conclude that $W$ is $H$-zero if and only if for any $W'$ as above we have that
$$W'^{-1} W H_*(W')$$
has $\cal G$-length 0 (which is equivalent to being contained in some vertex group).

For a formal proof of the last conclusion we observe that $g' := W'^{-1} W_1 \in G_v$ implies $W'' = W'^{-1} W = W'^{-1} W_1 g H_*(W_1)^{-1} = g' g H_*(W_1)^{-1}$ and hence $W'' H_*(W') = g' g H_*(W_1)^{-1} H_*(W') = g' g H_*(W_1^{-1} W') = g' g H_*(W_1^{-1} W') = g' g H_*(g'^{-1})$, where $H_*(g'^{-1}) \in G_v$ follows from our assumption $H_\Gamma = id_{\Gamma(\cal G)}$.
\end{rem}

\section{Quotient graph-of-groups isomorphism}

\subsection{Quotient graph-of-groups}

${}^{}$

Let $\cal{G}$ be a graph-of-groups and $\cal{G}_{0}$ be a sub-graph-of-groups of  
$\cal G$.
By sub-graph-of-groups we simply mean the restriction of $\cal{G}$ to a connected subgraph\footnote[1]{\ Be aware that this definition of sub-graph-of-groups is different from the one in \cite{Bass93}.}. Denote $\Gamma=\Gamma(\cal{G})$ and $\Gamma_0=\Gamma(\cal{G}_{0})$; by definition $\Gamma_{0}$ is a connected subgraph of $\Gamma$.

\medskip

We will denote 
by
$\bar{\cal{G}}= \cal{G} / \cal{G}_0$ the quotient graph-of-groups of $\cal{G}$ by $\cal{G}_{0}$, which we define now in detail:

The graph underlying $\bar{\cal{G}}$, denoted by $\bar{\Gamma}=\Gamma /\Gamma_{0}$ with 
$V(\bar{\Gamma})=(V(\Gamma)\smallsetminus V(\Gamma_0)) \cup \{V_0\}$ 
and $E(\bar{\Gamma})=E(\Gamma)\smallsetminus E(\Gamma_0)$, is obtained precisely by contracting $\Gamma_0$ into a vertex $V_0$ through the map:
\begin{align*}
q: \Gamma & \rightarrow \bar{\Gamma} \\
    x &  \mapsto V_{0} \text{\ \ \ if x $\in$ $E(\Gamma_0)$ or $V(\Gamma_0)$ } \\
    x & \mapsto X \text{\ \ \ otherwise } 
\end{align*}
Thus on
$\Gamma \smallsetminus \Gamma_0$ 
the map $q$ sends $x$, by which we mean 
either an edge or a vertex, to its natural correspondence in $\bar{\Gamma}$. We denote 
$q(x)$ 
by the same but capitalized letter 
$X \in \bar{\Gamma}$, to avoid confusion.

We choose a vertex 
\begin{equation}\label{choice1}
p_{0} \in V(\Gamma_{0})
\end{equation} 
as basepoint and set 
$G_{V_0}=\pi_{1}(\cal{G}_0,p_{0})$.

For every $v \in V(\Gamma)$ and $V=q(v) \neq V_0$, we set $G_{V}=G_{v}$. For every $e \in E(\Gamma)$ and $E=q(e) \in E(\bar{\Gamma})$, let $G_{E}=G_{e}$, and $f_{E}=f_{e}$ if $\tau(E) \neq V_{0}$. If $\tau(E)= V_{0}$, we choose a word 
\begin{equation}\label{choice2}
\gamma_{E} \in \Pi(\cal{G}_{0}) \subset \Pi(\cal{G})
\end{equation} 
from 
$p_{0}$ to $\tau(e) \in V(\Gamma_0)$ 
and define the edge homomorphism of $E$ to be 
$f_{E} := ad_{\gamma_{E}} \circ f_{e}$.

\medskip

We define 
a 
group homomorphism 
$\theta$ from $\Pi(\bar{\cal{G}})$ to $\Pi(\cal{G})$ 
via:
\begin{align*}
\theta: \Pi(\bar{\cal{G}})  & \rightarrow \Pi(\cal{G}) \\
	t_{E} &  \mapsto  t_{e} \gamma_{E}^{-1} \text{\ \ \ \ if $\tau(E)=V_0$, $\tau(\bar{E}) \neq V_0$ } \\
	t_{E} &  \mapsto  \gamma_{\bar{E}} t_{e} \gamma_{E}^{-1} \text{\ \ \ \ if $\tau(E)=\tau(\bar{E})=V_0$ } \\
    \theta & \text{\ acts as identity elsewhere } 
\end{align*}

Through the identification 
$G_{V_0}= \pi_{1}(\cal{G}_0, p_0)$
we 
see immediately that the restriction of $\theta$ 
to 
$\pi_1(\bar{\cal{G}}, V_0)$ defines 
an isomorphism 
from $\pi_1(\bar{\cal{G}}, V_0)$ to 
$\pi_1(\cal{G}, p_0)$
which we denote by 
$\theta_0$: 

The 
inverse map 
$\theta_0^{-1}$ 
is given  through introducing a cancelling pair 
$\gamma_E^{-1} \gamma_E$
after 
the stable letter $t_e$ for any edge $e$ 
with $\tau(e) \in V(\Gamma_0)$, and 
$\gamma_{\bar E}^{-1} \gamma_{\bar E}$
before 
the stable letter $t_e$ for any edge $e$ 
with $\tau(\bar e) \in V(\Gamma_0)$.
One then maps
$t_e$ to $\gamma_{\bar{E}}^{-1} t_{E} \gamma_{E}$ if both $\tau(e), \tau(\bar{e}) \in V(\Gamma_0)$, 
and one maps 
$t_e$ to $t_{E}\gamma_{E}$ if only $\tau(e) \in V(\Gamma_0)$.

\begin{rem}
\label{for-A2}
${}^{}$
For later purposes the reader should note here that 
for any edge $E = q(e)$ of $\Gamma(\bar{\cal{G}})$ with $\tau(E)=V_{0}$ there exist a vertex $v(E) := \tau(e)\in V(\cal{G}_{0})$ such that for the
``connecting
word'' 
$\gamma_{E} \in \Pi(\cal{G}_0)$ 
from 
$p_{0}$ 
to $v(E)$ 
one has $\theta \circ f_{E}(G_{E}) \subset \gamma_{E} G_{v(E)} \gamma_{E}^{-1}$.
\end{rem}

\subsection{Quotient graph-of-groups isomorphism}

${}^{}$

In this subsection we define the notion of a
{\it quotient graph-of-groups isomorphism}.

The graphs-of-groups $\cal{G}$, $\cal{G}_0$, $\bar{\cal{G}}$ and 
the group homomorphisms $\theta$ and $\theta_0$
are defined as in the previous subsection.
In particular, let 
$V_0, \gamma_E$ and $p_0$ 
be as given there.

\medskip

Let $H: \cal{G} \rightarrow \cal{G}$ be a graph-of-groups isomorphism which acts as identity on the graph $\Gamma$. 
The map 
$H_{0}: \cal{G}_{0} \rightarrow \cal{G}_{0}$, 
obtained by restricting $H$ 
to
$\cal{G}_0$,
is called the {\it local graph-of-groups isomorphism}.

In order to
define $\bar{H}: \bar{\cal{G}} \rightarrow \bar{\cal{G}}$ 
we set:

\begin{enumerate}
\item 
$\bar{H}_{V_0}=H_{0, *p_{0}}: \pi_{1}(\cal{G}_0, p_0) \rightarrow \pi_{1}(\cal{G}_0, p_0)$; 
\item 
$\delta(E)= H_{*}(\gamma_{E}) \delta(e) \gamma_{E}^{-1} $,
for all $E$ such that $\tau(E)=V_{0}$;
\item $\bar{H}$ 
``equals''
$H$ on the rest of $\bar{\cal{G}}$
(modulo 
replacing 
$x$ 
by 
$X$ 
as 
explained in 
the 
previous subsection).
In particular, 
$\bar H_{\bar{\Gamma}}$
is the identity on the quotient graph 
$\bar{\Gamma}$.
\end{enumerate}

\begin{prop}\label{quotientisomorphism}
The above conditions (1) - (3) give a well defined graph-of-groups isomorphism
$\bar{H}: \bar{\cal G} \to \bar{\cal G}$. 
It 
induces an outer automorphism 
$\hat{\bar H}$ which is 
conjugate to $\hat{H}$ 
via 
the isomorphism 
$\theta_0: \pi_{1}(\bar{\cal{G}}, V_0) \rightarrow \pi_{1}(\cal{G}, p_0)$.
\end{prop}

\begin{proof}
In order 
to show
that 
the above data (1) - (3) give a 
well defined 
graph-of-groups isomorphism, 
the only non-trivial step 
to verify is 
the condition
that $ \bar{H}_{V_{0}} \circ f_{E}=ad_{\delta(E)} \circ f_{E} \circ \bar{H}_{E}$ for all $E=q(e)$ with $\tau(E)=V_{0}$. Denote $v=\tau(e) \in V(\Gamma_0)$.

Observe
first that for 
any 
$h \in G_{v}$, since $H_{0}$ induces an isomorphism on $\Pi(\cal{G}_{0})$, we 
have:
\begin{align*}
\bar{H}_{V_{0}}(\gamma_{E} h \gamma_{E}^{-1})
=H_{0*}(\gamma_{E}) H_{v}(h) H_{0*}(\gamma_{E}^{-1})
=H_{*}(\gamma_{E}) H_{v}(h) H_{*}(\gamma_{E}^{-1})
\end{align*}

For 
any 
$g \in G_{E}$
we 
compute, where the fourth equality uses the previous observation, for $h = f_e(g)$:
\begin{align*}
ad_{\delta(E)} \circ f_{E} \circ \bar{H}_{E}(g)
& = H_{*}(\gamma_E) \delta(e) \gamma_E^{-1} f_{E}(\bar{H}_{E}(g)) \gamma_E \delta(e)^{-1} H_{*}(\gamma_E^{-1})\\
& = H_{*}(\gamma_E) \delta(e) f_{e}(H_{e}(g)) \delta(e)^{-1} H_{*}(\gamma_E^{-1})\\
& = H_{*}(\gamma_E) H_{v}(f_{e}(g)) H_{*}(\gamma_E^{-1})\\
& = \bar{H}_{V_{0}}(\gamma_E f_{e}(g) \gamma_E^{-1}) \\
& = \bar{H}_{V_{0}} \circ f_{E}(g)
\end{align*}

In order to illustrate the previous proof we propose the diagram below: 
every face of the ``cube'' 
pictured 
there 
commutes, up 
to inner automorphisms
on the front and back faces, defined by the correction terms $\delta(E)$ and $\delta(e)$ 
respectively.

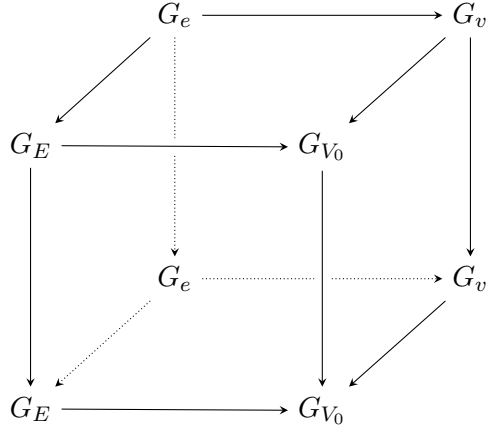
\begin{center}
\begin{tikzpicture}
  \matrix (m) [matrix of math nodes, row sep=3em,
    column sep=3em]{
    & G_{e} & & G_{v} \\
    G_{E} & & G_{V_{0}} & \\
    & G_{e} & & G_{v} \\
    G_{E} & & G_{V_{0}} & \\};
  \path[-stealth]
    (m-1-2) edge (m-1-4) edge (m-2-1)
            edge [densely dotted] (m-3-2)
    (m-1-4) edge (m-3-4) edge (m-2-3)
    (m-2-1) edge [-,line width=6pt,draw=white] (m-2-3)
            edge (m-2-3) edge (m-4-1)
    (m-3-2) edge [densely dotted] (m-3-4)
            edge [densely dotted] (m-4-1)
    (m-4-1) edge (m-4-3)
    (m-3-4) edge (m-4-3)
    (m-2-3) edge [-,line width=6pt,draw=white] (m-4-3)
            edge (m-4-3);
\end{tikzpicture}
\end{center}
Here the map $G_e \rightarrow G_E$ is the identity, while the map $G_v \rightarrow G_{V_0}$ is 
$ad_{\gamma_E}$.

\medskip
After having thus proved the first sentence of the proposition, we now turn to the second:
we want to show that the following diagram 
commutes.
\[
\begin{CD}
\pi_1(\bar{\cal{G}},V_0) @>\bar{H}_{*V_0}>> \pi_1(\bar{\cal{G}},V_0) \\
@V\theta_0 VV @VV\theta_0 V \\
\pi_1(\cal{G}, p_0)  @>>H_{*p_0}> \pi_1(\cal{G}, p_0) 
\end{CD}
\]

Notice that the group homomorphism $\theta$ acts on all vertex groups other than $G_{V_0}$ as identity. On 
the 
other hand,
for $E \in E(\bar{\cal{G}})$ with $\tau(E)=V_0$ but $\tau(\bar{E}) \neq V_0$
we 
have:
\begin{align*}
\theta \circ \bar{H}_{*}(t_{E}) = \theta( \delta(\bar{E}) t_{E} \delta(E)^{-1} )
& = \delta(\bar{E}) t_{e}\gamma_{E}^{-1} \delta(E)^{-1} \\
& = \delta(\bar{e}) t_{e} \delta(e)^{-1} H_{*}(\gamma_E)^{-1}\\
& = H_{*}(t_{e} \gamma_E^{-1})\\
& = H_{*} \circ \theta(t_{E})
\end{align*}
A similar computation applies to
$E \in E(\bar{\cal{G}})$ with $\tau(E)=\tau(\bar{E})=V_0$.

Therefore 
we obtain 
$\theta_0 \circ \bar{H}_{*V_{0}}=H_{*p_{0}} \circ \theta_0$,
and hence $\hat{\bar{H}}$ and $\hat{H}$ are conjugate to each other
through the outer isomorphism 
$\hat \theta_0$.
\end{proof}

\begin{rem}
\label{for-A3}
Note in particular, for all $E$ with $\tau(E)=V_{0}$, we have by definition that $\delta(E)$ is 
$H^{-1}_{0}$-conjugate to an element with 
$\cal G_0$-length
equal to zero, 
and hence is $H^{-1}_{0}$-zero:
The word 
$\gamma_{E}$
satisfies 
$$ H_{0*}(\gamma_{E}^{-1}) \theta_{}(\delta(E)) \gamma_{E} \in G_{v(E)} \, .$$
\end{rem}

\begin{rem}\label{uniqueness of quotient automorphism}
Our formal construction of quotient graph-of-group 
isomorphism 
$\bar H$, 
constructed as above,  
depends on the choice of base point 
$p_{0}$
in $\Gamma_0$ and of 
``connecting words"
$\gamma_{E}$
for all edges 
$E$ 
with 
$\tau(E)=V_0 $,
as set up in (\ref{choice2}) in order to define $\theta$ and $\theta_0$. 

However, 
the 
outer automorphism $\hat{\bar{H}}$ 
induced by
$\bar H: \bar{\cal{G}} \rightarrow \bar{\cal{G}}$
depends neither on the choice of 
the 
$\gamma_E$
nor on the choice of 
$p_0$, up to conjugation by the natural isomorphism
$\theta^{-1}_0 \theta'_0$, where $\theta'_0$ is the map analogous to $\theta_0$ defined through an alternative choice of $p_0$ and and the $\gamma_E$.

This is a direct consequence of the statement in Proposition \ref{quotientisomorphism} that $\hat{\bar H}$ is conjugate to $\hat H$
via $\theta_0$.

${}^{}$
Alternatively, a direct proof, without passing through $\hat H$, can be given by applying Lemma \ref{basiclemma}; this yields the slightly stronger result that a second quotient automorphism $\bar H': \bar{\cal G}' \to \bar{\cal G}'$ is conjugated to $\bar H$ by a graph-of-groups isomorphism $F: \bar{\cal G} \to \bar{\cal G}'$.
\end{rem}

\begin{rem}
We may apply this quotient procedure above on several disjoint connected 
subgraphs-of-groups
of $\cal{G}$ and obtain 
the 
analoguous 
conclusion that the quotient graph-of-groups isomorphism 
$\bar{H}$ 
is well defined and 
induces an outer automorphism conjugate to $\hat{H}$.
\end{rem}

\begin{rem} 
\label{pair-notation}
For simplicity of notations, we sometimes represent 
the simultaneous quotienting of the graph-of-groups $\cal G$ and of the isomorphism $H$ by referring to the {\em quotient pair} $(\bar H, \bar{\cal G})$,
obtained from $(H, \cal{G})$ 
{\em modulo the pair $(H_0, \cal{G}_0)$}.
\end{rem}

\section{Blowing up graph-of-groups automorphism}\label{blow up}

In this section, we will reverse the quotient construction in the previous section; we emphasize this 
reversal 
by the choice of our notation.

For simplicity of the presentation, we only give the blow-up construction at a single vertex. However (for example through iterating this procedure), one can generalize the technique described in this section directly to a blow-up construction at several vertices simultaneously
to 
obtain 
Theorem~\ref{intro1} in the 
Introduction.

\subsection {Assumptions}

${}^{}$

Let $\bar{H}: \bar{\cal{G}} \rightarrow \bar{\cal{G}}$
and 
$H_0 : \cal{G}_0 \rightarrow \cal{G}_0$ 
be 
graph-of-groups 
automorphisms which act as the identity on their underlying 
graphs 
$\bar{\Gamma}$ and $\Gamma_0$
respectively.
Let $V_0$ be a vertex of 
$\bar{\Gamma}$. We 
consider 
the following assumptions:

\begin{enumerate} 

\item[(A1)] 
There exist a vertex 
$p_{0} \in V(\cal{G}_{0})$
and a group isomorphism 
$\theta_{0}: G_{V_{0}} \rightarrow \pi_{1}(\cal{G}_{0}, p_{0})$
such that 
$$\theta_{0} \circ H_{V_{0}}=H_{0*,p_{0}} \circ \theta_{0};$$

\item[(A2)] {\it Compatibility requirement for graph-of-groups}: 

For any edge $E$ of $\Gamma(\bar{\cal{G}})$ with $\tau(E)=V_{0}$ there exist some 
vertex $v(E)\in V(\cal{G}_{0})$ and a 
``connecting word'' 
$\gamma_{E} \in \Pi(\cal{G}_0)$ from 
$p_{0}$
to $v(E)$ such that $\theta_{0} \circ f_{E}(G_{E}) \subset \gamma_{E} G_{v(E)} \gamma_{E}^{-1}$;

\item[(A3)] {\it Compatibility requirement for isomorphism}:

The word $\gamma_{E}$ also satisfies 
$ H_{0*}(\gamma_{E}^{-1}) \theta_{0}(\delta(E)) \gamma_{E} \in G_{v(E)}$.

\end{enumerate}

\begin{defn}
\label{defn-6.1}
Assume that condition (A1) is satisfied. Then
the pair $(H_0, \cal{G}_0)$, which is called the {\it local graph-of-groups isomorphism associated to $V_0$}, is said to be {\it compatible} with $(\bar{H}, \bar{\cal{G}})$ if both compatibility requirements (A2) and (A3) are 
satisfied.

\end{defn}

\begin{rem}
\label{rem6.1}
${}^{}$
We'd like to note:
\begin{enumerate}
\item The compatibility requirement for isomorphism implies, for any $\delta(E)$ with $\tau(E)=V_{0}$, that the element $\theta_{0}(\delta(E))$ is $H^{-1}_{0}$-zero. 

Conversely, in order to derive the compatibility requirement for isomorphism for any correction term $\delta(E)$ 
such that $\theta_{0}(\delta(E))$ is $H^{-1}_{0}$-zero,
one needs the additional hypothesis that $\delta(E)$ is {\it $G_{E}$-compatible}: by this we 
mean
that there is a connected word $\gamma_{E} \in \Pi(\cal{G})$ which satisfies both 
assumptions (A2) and (A3)
above.

\item For each $E$ with $\tau(E)=V_{0}$, there may exist more then one pair of $(v(E), \gamma_{E})$ such that the above conditions hold.

\end{enumerate}
\end{rem}

\subsection{Existence of 
the 
blow-up}\label{blow-up}

\begin{thm}\label{blowup}
Let 
$\bar{H}: \bar{\cal{G}} \rightarrow \bar{\cal{G}}$ and $H_0 : \cal{G}_0 \rightarrow \cal{G}_0$
be graph-of-groups isomorphisms
which act as the identity on their underlying graphs, and 
let 
$V_0 \in V(\bar{\cal{G}})$ 
be a vertex for which 
condition (A1) is satisfied. 

Then there exists a ``blow-up'' 
graph-of-groups $\cal G$ 
and a ``blow-up'' 
graph-of-groups isomorphism 
$H: \cal{G} \rightarrow \cal{G}$, 
which contains 
$H_0$ 
as local graph-of-groups isomorphism and 
yields 
modulo $(H_0, \cal{G}_0)$ the quotient pair $(\bar{H}, \bar{\cal{G}})$, 
if and only if the conditions (A2) and (A3) are 
satisfied.

In particular, $H$ and $\bar H$ induce outer automorphisms which are conjugate.
\end{thm}

\begin{proof}
For the convenience of the reader we divide this proof in 6 steps; in the first two we present the data which will serve to define $\cal G$ and $H$ respectively.

\smallskip
\noindent
(1)
The graph
$\Gamma(\cal{G})$,
with vertex set $V(\cal{G}) = V(\bar{\cal{G}}) \setminus \{V_0\} \cup V(\cal{G}_0)$, 
is obtained from $\Gamma(\cal{G})$ and $\Gamma(\cal{G}_0)$ by replacing every edge
$E$
with terminal vertex 
$\tau(E) = V_{0}$
by an edge 
$e$
with 
terminal vertex $v(E)$.
The analogous replacement is done for $\bar{E}$. 
If $E$ has both endpoints distinct from $V_0$ we leave them as they are, but rename $E$ by the corresponding small letter $e$.

We set
$ G_{e}=G_{E}$ and define, 
if $\tau(e) = v(E)$, the edge injection by 
$f_{e}(g)= \gamma_{E}^{-1} \theta_{0}(f_{E}(g)) \gamma_{E}$, for every $g \in G_{e}$. 
For $\tau(e) = \tau(E)$ we define $f_e = f_E$.

\smallskip
\noindent
(2)
The isomorphism 
$H:\cal{G} \rightarrow \cal{G}$ 
is equal 
to $H_{0}$ 
or to $\bar{H}$ 
when restricted 
to 
$\cal{G}_{0}$ 
or to 
$\bar{\cal{G}} \backslash \{V_{0}\}$ respectively, except that 
in the case $\tau(e) = v(E)$ we modify the correction term to 
$\delta(e)=  H_{0*}(\gamma_{E}^{-1}) \theta_{0}(\delta(E)) \gamma_{E}$.

\smallskip
\noindent
(3)
In order to show that the data defined above in (1) 
give 
a well defined graph-of-groups one only needs to verify that for any edge $e$ of $\cal G$ the edge injection $f_e$ has its image in the vertex group $G_{\tau(e)}$. If $\tau(e)$ is not contained in $V(\cal G_0)$, then this is immediate from the definition. If $\tau(e) \in V(\cal G_0)$, then we use 
the compatibility 
requirement
for 
graph-of-groups, which gives 
(A2), 
$f_{e}(G_{e})= \gamma_{E}^{-1} \theta_{0}(f_{E}(G_{E})) \gamma_{E} \subset G_{v(E)}$. 

\smallskip
\noindent
(4)
We now want to show that the data defined above in (2) 
give 
a 
well-defined
graph-of-groups isomorphism. This is equivalent to showing for every edge $e$ of $\cal G$ the equality stated in condition (4) of Definition \ref{graphofgroupsiso}. Again, for $\tau(e) \notin V(\cal G_0)$ this is a direct consequence of our set-up. For $\tau(e) \in V(\cal G_0)$ we compute (where the third equality uses the definition of $f_e$ from (2) above, and the fifth equality uses condition (A1)):
\begin{align*}
ad_{\delta(e)} \circ f_{e} \circ H_{e}(g)
& = \delta(e) f_{e}(H_{e}(g)) \delta(e)^{-1} \\
& = H_{0*}(\gamma_{E}^{-1})\theta_{0}(\delta(E)) \gamma_{E} f_{e}(H_{e}(g)) \gamma_{E}^{-1}\theta_{0}(\delta(E)^{-1}) H_{0*}(\gamma_{E}) \\
& = H_{0*}(\gamma_{E}^{-1})\theta_{0}(\delta(E) f_{E}(H_{E}(g))\delta(E)^{-1}) H_{0*}(\gamma_{E})\\
& = H_{0*}(\gamma_{E}^{-1})\theta_{0}(H_{V_{0}}(f_{E}(g))) H_{0*}(\gamma_{E})\\
\ssh
& = H_{0*}(\gamma_{E}^{-1})H_{0*, P_0}(\theta_{0}(f_{E}(g))) H_{0*}(\gamma_{E})\\
& = H_{0*}(\gamma_{E}^{-1} \theta_{0}(f_{E}(g)) \gamma_{E}) = H_{0*}(f_{e}(g))
= H_{v(E)} \circ f_{e}(g)
\end{align*}

\smallskip
\noindent
(5)
We now observe from our construction above that
the blow-up 
graph-of-groups isomorphism $H$ contains $H_0$
as local graph-of-groups isomorphism. We furthermore have already a base point 
$p_0$ 
as well as connecting words 
$\gamma_E$ from $p_0$ 
to $\tau(e)$ for each edge $e$ with terminal vertex $\tau(e) \in V(\cal G_0)$ specified, so that one can readily apply Proposition \ref{quotientisomorphism} to obtain a quotient graph-of-groups isomorphism, which, since 
$p_0$ and all $\gamma_E$
are as chosen before, must agree with the isomorphism $\bar H: \bar{\cal G} \to \bar{\cal G}$.

In particular it follows that the induced outer automorphisms $\hat H$ and $\hat{\bar H}$ are conjugate.

\smallskip
\noindent
(6)
Finally, it follows from Remarks \ref{for-A2} and \ref{for-A3} that any blow-up pair $(H, \cal G)$ which quotients modulo $(H_0, \cal G_0)$ to the given pair $(\bar H, \bar{\cal G})$ must necessarily satisfy the 
conditions (A2) and (A3).
\end{proof}

\section{Partial Dehn Twist case}\label{partialDehntwist}

In this section we will apply 
Theorem 
\ref{blowup} to the special case where all edges $E$ with terminal vertex $V_0$ have trivial edge group:
$$G_E = \{1\} \qquad \text{if} \qquad  \tau(E) = V_0$$
In this case the compatibility conditions from Definition \ref{defn-6.1} simplify considerably, as condition (A2) is trivially satisfied. Regarding the other compatibility condition, we observe that a connecting word $\gamma_E$ which satisfies condition (A3) exists, if and only if the $\theta_0$-image of the correction term $\delta(E)$ is $H_0^{-1}$-zero, in the terminology of section \ref{Hconj}. Hence we define:

\begin{defn}
\label{locally-zero}
Let $\bar{H}: \bar{\cal{G}} \rightarrow \bar{\cal{G}}$ and $H_0 : \cal{G}_0 \rightarrow \cal{G}_0$ be graph-of-groups isomorphisms which act as the identity on their underlying graphs, and let $V_0 \in V(\bar{\cal{G}})$ be a vertex 
with an isomorphism $\theta_0: G_{V_0} \to \pi_1(\cal G_0, p_0)$ as in condition (A1).

Then an edge $E$ of $\bar{\cal G}$ is said to be {\em locally zero} if the correction term of $E$ has image $\theta_0(\delta(E))$ which is $H_0^{-1}$-zero.
In other words, there exists a connected word $\gamma_E \in \Pi(\cal G_0)$ such that 
$ H_{0*}(\gamma_{E}^{-1}) \theta_{0}(\delta(E)) \gamma_{E} \in G_{v}$ 
for some vertex $v$ of $\cal G_0$.
\end{defn}

We will now specialize further to the case of partial Dehn twists $\bar{H}: \bar{\cal{G}} \rightarrow \bar{\cal{G}}$ as defined in Definition \ref{partial Dehn twist}, which have trivial vertex group isomorphisms $\bar H_V$ except at vertices 
$V_i$ that belong to a subset $\cal V_0 \subset V(\bar{\cal G})$. At those ``special'' vertices we want to assume furthermore that 
$\bar H_{V_i}$ is a Dehn twist automorphism, given 
via an isomorphism $\theta_i: G_{V_i} \to \pi_1(\cal G_i, p_i)$ as in condition (A1)
by some {\em local} Dehn twist $D_i: \cal G_i \to \cal G_i$ 
(see 
Definition \ref{dehn}), so that $\bar H$ is a 
{\em partial Dehn twist relative to a family of local Dehn twists}.

\begin{cor}\label{Dehn01}
Let $\bar{H}: \bar{\cal{G}} \rightarrow \bar{\cal{G}}$ be a partial Dehn twist relative to $\cal{V}_0 \subset V(\bar{\cal{G}})$,
and assume
that for each 
$V_i \in \cal{V}_0$ the 
vertex 
group isomorphism $H_{V_i}$ is a Dehn twist automorphism 
represented by a 
``local'' 
Dehn twist 
$D_i: \cal{G}_i \to \cal G_i$.

Then one can blow up $(\bar{H}, \bar{\cal{G}})$ via $(D_i, \cal{G}_i)$ to 
obtain a 
Dehn twist $D: \cal{G} \rightarrow \cal{G}$ if and only if
every edge $E$ of $\bar{\cal G}$ with $\tau(E) \in \cal{V}_0$ is locally zero.

In this case the induced outer automorphism
$\hat{\bar{H}}: \pi_1(\bar{\cal{G}}) \rightarrow \pi_1(\bar{\cal{G}})$ is a Dehn twist 
automorphism.
\end{cor}

\begin{proof}
Let $V_i$ be a vertex contained in $\cal{V}_0$.
Since 
any edge $E$
terminating at 
$V_i$
has trivial edge group, the compatibility requirement 
for
graph-of-groups
(A2) 
holds automatically. 
On the other hand, the 
assumption that any edge $E$ with $\tau(E) \in \cal{V}_0$ is locally zero 
is equivalent to
the compatibility requirement 
for
isomorphism
(A3). 
We can hence 
apply 
Theorem 
\ref{blowup} to directly obtain the desired ``if and only if'' statement. 

Since $\bar H$ is a partial Dehn twist and the $D_i$ are Dehn twists, it follows directly 
from the construction of the blow-up isomorphism in the proof of Theorem \ref{blowup} that $D$ satisfies 
the first three properties of
Definition \ref{dehn}.

In order to see that the fourth condition of Definition \ref{dehn} is also satisfied, we consider three cases: If an edge $E$ from $\bar{\cal G}$ has endpoint outside of $\cal V_0$, then all relevant data for $E$ and the corresponding edge $e$ in $\cal G$ coincide, so that we can simply use the last condition of Definition \ref{partial Dehn twist}. If $\tau(E) \in \cal V_0$, then $G_E = \{1\}$ follows from Definition \ref{partial Dehn twist}, so that the fourth condition of Definition \ref{dehn} is automatically satisfied. Finally, if the edge $e$ from $\cal G$ in question is an edge of some of the local graph-of-groups $\cal G_i$, then we use that $D_i: \cal G_i \to \cal G_i$ itself is assumed to satisfy Definition \ref{dehn}, so that in particular its fourth condition of this definition holds for $e$.

As a consequence we obtain
from 
Proposition 
\ref{equivalent Dehn twist}
that $\hat D$ is a Dehn twist automorphism. But Theorem \ref{blowup} also states that the outer automorphisms induced by $\bar H$ and by $D$ are conjugate, which shows that $\hat{\bar H}$ is also a Dehn twist automorphism.
\end{proof}

Corollary~\ref{Dehn01} 
gives the possibility to decide  
the existence of a blow-up Dehn twist relative to particular 
given 
local Dehn twist representatives. A harder but more interesting question is 
whether 
one can 
blow up the given partial Dehn twist $\bar H: \bar{\cal G} \to \bar{\cal G}$ to a
global Dehn twist with respect to 
{\em some} family of
local Dehn twists
that induce the vertex automorphisms of $\bar H$. In other words
(using the terminology of Definition \ref{partial-D-rel-D}): 

${}^{}$
``When 
does 
a partial Dehn twist with Dehn twist automorphism on the vertices induce a Dehn twist 
automorphism~?''

The subtlety of this question is illustrated by the 
Examples 
\ref{no-blow-up}
and
\ref{blow-up-possible}
in the Introduction.
A complete answer is given in \cite{KY02}.

\begin{rem}
\label{algorithmics}
Assume that in the situation of Corollary \ref{Dehn01} the following data are given:
\begin{enumerate}
\item[(a)] The graph-of-groups $\bar{\cal G}$ and $\cal G_i$ have free groups of finite rank 
as vertex and edge groups, with chosen bases for each of them, and 
that the edge maps are given in the usual fashion by 
specifying the images of the edge groups basis elements as words in the basis of the adjacent vertex group.
\item[(b)] The graph-of-groups automorphisms $\bar H$ and $H_i$ are similarly given by specifying the images of the 
given 
basis elements as word in those bases, and by specifying for every edge $E$ of $\bar{\cal G}$ 
with terminal vertex $V_i \in \cal V_0$
the $\theta_i$-image of the correction term $\delta(E)$ as word $W(\delta(E)) \in \Pi(\cal G_i)$.
\item[(c)] For every edge $E$ with $\tau(E) = V_i \in \cal V_0$, a vertex $v_E$ of $\cal G_i$ and connecting words $\gamma_E \in \Pi(\cal G_i)$ are specified which satisfy 
$$ H_{i*}(\gamma_{E}^{-1}) W(\delta(E)) \gamma_{E} \in G_{v_E}.$$
\end{enumerate}
Then from these data one derives directly (in an algorithmic way) the analogous data needed to define the blow-up graph $\cal G$ and the blow-up 
isomorphism 
$H$. 
Indeed, the precise instructions for this procedure are given in the parts (1) and (2) of proof of Theorem \ref{blowup}.
\end{rem}

\end{document}